\documentclass{amsart}
\usepackage[T1]{fontenc}
\usepackage[hidelinks]{hyperref}
\usepackage{amssymb}
\usepackage{eurosym}
\usepackage{amsfonts}
\usepackage{amsmath}
\usepackage{geometry}
\usepackage[usenames,dvipsnames]{xcolor}
\usepackage{enumerate}
\usepackage{todonotes}
\usepackage{mathrsfs}
\usepackage{enumitem}
\usepackage{color}

\newif\ifdraft
\draftfalse
\newtheorem{theorem}{Theorem}
\theoremstyle{plain}

\newtheorem{corollary}[theorem]{Corollary}

\newtheorem{definition}[theorem]{Definition}

\newtheorem{lemma}[theorem]{Lemma}

\newtheorem{proposition}[theorem]{Proposition}
\newtheorem{remark}[theorem]{Remark}

\numberwithin{equation}{section}

\def\D{\mathscr{D}}

\def\R{\mathbb{R}}

\def\ND{\zeta}

\def\W{{\mathbb W}}
\def\WW{{\mathbf W}}

\def\ZZ{{\mathbf Z}}

\begin{document}
\def\cprime{$'$}
\def\cprime{$'$}

\title[Variational rough PDEs]{Existence, uniqueness and stability of semi-linear rough partial differential equations}

\author{Peter K. Friz}
       \address{TU Berlin, Institut f\"ur Mathematik, MA 7-2\\
Strasse des 17. Juni 136\\
10623 Berlin\\
Germany\\
and\\
Weierstrass-Institut f\"ur Angewandte Analysis und Stochastik \\
Mohrenstrasse 39 \\
10117 Berlin \\
Germany. Financial support by the DFG via Research Unit FOR 2402 is gratefully acknowledged.}
       \email{friz@math.tu-berlin.de, friz@wias-berlin.de}

\author{Torstein Nilssen}
       \address{TU Berlin, Institut f\"ur Mathematik, MA 7-2\\
Strasse des 17. Juni 136\\
10623 Berlin\\
Germany. Financial support by the DFG via Research Unit FOR 2402 is gratefully acknowledged.}
       \email{nilssen@math.tu-berlin.de}

\author{Wilhelm Stannat}
       \address{TU Berlin, Institut f\"ur Mathematik, MA 7-2\\
Strasse des 17. Juni 136\\
10623 Berlin\\
Germany. Financial support by the DFG via Research Unit FOR 2402 is gratefully acknowledged.}
       \email{stannat@math.tu-berlin.de}

\begin{abstract}
We prove well-posedness and rough path stability of a class of linear and semi-linear rough PDE's on $\R^d$ using the variational approach.
This includes well-posedness of (possibly degenerate) linear rough PDE's in $L^p(\R^d)$, and then -- based on a new method -- energy estimates for non-degenerate linear rough PDE's. We accomplish this by controlling the energy in a properly chosen weighted $L^2$-space, where the weight is given as a solution of an associated backward equation. These estimates then allow us to extend well-posedness for linear rough PDE's to semi-linear perturbations.
\end{abstract}

\maketitle

\section{Introduction}

In this paper we use the variational approach to prove well-posedness for a class of linear and semi-linear rough PDE's on $\R^d$ of the following type,
\begin{equation} \label{u_introSemiLinear}
du_{t}  =\left[ Lu_{t} + F(u_t) \right] dt +\Gamma  u_{t}d\mathbf{W}_{t} \ , \hspace{.5cm} u_0 \in L^p(\R^d) \ ,
\end{equation}
where $L$ and $\Gamma$ are (linear) differential operators of second (resp. first) order as detailed in \eqref{SecondOrder},\eqref{FirstOrder} below,  $\WW$ a (geometric) rough path and, with focus on the $L^2$-scale a Lipschitz non-linearity of the form 
$$
    F: H^1(\R^d) \to L^2 (\R^d) \ ,
$$
which allows for non-linear dependence on $\nabla u$. Integrated in time, the above equation reads
\begin{equation*}
u_t -u_s
=\int_{s}^{t} \left[ Lu_{r} + F(u_r) \right] dr+\int_{s}^{t}\Gamma  u_{r}d\mathbf{W}_{r},
\end{equation*}%
provided $u$ is sufficiently regular (in space) such as to make $Lu,\Gamma
 u$ and $F(u)$ meaningful, \textit{and} provided the last term makes sense as
rough integral. Since we do not expect our solutions to be regular in space,  $Lu$ and $\Gamma u$ are understood in the weak sense. More precisely, (\ref{u_introSemiLinear}) means that, for all $s<t$, on some time horizon $[0,T]$,
\begin{equation*}
( u_{t},\phi ) -( u_{s},\phi )
=\int_{s}^{t} \left[ ( u_{r},L^{\ast }\phi ) + (F(u_r), \phi) \right] dr+\int_{s}^{t}( u_{r},\Gamma ^{\ast }\phi ) d\mathbf{%
W}_{r},
\end{equation*}%
for suitable test functions $\phi$. We will use a non-degeneracy condition on $L$ to define $F(u)$ and the last term is a real valued rough integral.
Note that the rough integral will be defined as having the local expansion
\begin{equation} \label{introExpansion}
\int_{s}^{t} (  u_{r}, \Gamma^* \phi) d\mathbf{W}_{r} \simeq ( u_{s},\Gamma^*\phi)  (W_t - W_s)  + ( u_{s},\Gamma^* \Gamma^* \phi) \int_s^t (W_r - W_s) dW_r .
\end{equation}

Let us first consider the linear case, that is
\begin{equation} \label{u_introLinear}
du_{t}  =Lu_{t}dt  +\Gamma u_{t}d\mathbf{W}_{t},  \hspace{.5cm} u_0 \in L^p(\R^d) .
\end{equation}
This setting already contains the model problem $du_{t}  = \Delta u_{t}dt  + V \cdot \nabla u(t) d\mathbf{W}_{t}$, with $u_0 = u_0(x) \in L^2(\R^d)$ and a vector field $V=V(x)$, studied
in \cite{DGHT}, subsequently analyzed in detail in \cite{HH} with general second (resp. first) linear differential operators. Our methods are different, and we are instead able to adapt the rough path Feynman-Kac representation (cf. \cite{DFS}, \cite{FH14}) in terms of the diffusion process with generator $L$ and $\Gamma$ as well as the rough signal $\WW$.
In contrast to \cite{DFS}, we here consider the solution $u$ and the expansion \eqref{introExpansion} as functions in an appropriate Sobolev-space rather than pointwise. 

\medskip 
To show uniqueness, we introduce the corresponding backward equation
\begin{equation} \label{u_introLinearBackward}
-dv_{t}  =L^*v_{t}dt  +\Gamma^* v_{t}d\mathbf{W}_{t} ,  \hspace{.5cm} v_T \in L^q(\R^d) ,
\end{equation}
which again admits a rough path Feynman-Kac representation. Using a duality argument,  
\begin{equation} \label{DualTrick}
\int_{\R^d} u_T(x) v_T(x) dx =  \int_{\R^d} u_0(x) v_0(x) dx \ ,
\end{equation}
then allows us to infer uniqueness of the forward equation \eqref{u_introLinear} from existence of the backward equation \eqref{u_introLinearBackward}, and vice versa. We shall frequently jump between the two equations \eqref{u_introLinear} and \eqref{u_introLinearBackward} whenever one equation is more convenient. Note that this analysis is valid also when $L$ is degenerate. 
 
 \bigskip
 
In the case that where $L$ is non-degenerate, we specialize to the $L^2$-scale, and obtain energy estimates, similar to these of \cite{DGHT, HH}, but with different methods. It will not surprise readers familiar with Feynman-Kac theory for PDE's, that we require more regularity assumption than what a pure PDE approach (including \cite{HH}) requires. In turn, our construction yields fine-information about the stochastic characteristics in term of hybrid rough / It\^o diffusions, enables us solve an open problem  in the afore-mentioned works concerning the rough path stability in the natural function space where the solutions live. That is, by a direct analysis of the Feynman-Kac formula and a corresponding Lyapunov function, we can show continuity of the solution map as function from geometric rough path space $\mathcal{C}_g^\alpha([0,T])$, with ``Brownian'' roughness $\alpha \in (1/3,1/2]$, into
$$
       C([0,T]; L^2(\R^d)) \cap L^2([0,T]; W^{1,2}(\R^d)). 
$$
equipped with its natural Banach structure. (This is in contrast to \cite[Thm. 2]{HH} where compactness argument lead to some sub-optimal spaces in the continuity statement.)

\bigskip

 On a technical level, we rely on the existence of a suitable weight function $f= f_t(x)$, given as the solution of an associated backward rough PDE, such that
\begin{equation} \label{u_SquaredF}
d (u_t^2,f_t)  =  - ( |\sigma \nabla u_t|^2,  f_t) dt
\end{equation}
where $\sigma$ is a square root of the diffusion matrix and we note that there is no rough path term. 
In addition we show that $f$ is bounded away from $0$ and $\infty$ which allow us to infer from \eqref{u_SquaredF} the energy estimates for variational solutions. We note that a similar technique was applied in \cite{HNS} to obtain energy estimates for a class of rough PDE's of Burgers type. 

\newpage

With this precise linear solution theory in place, we then develop a novel (two-parameter) semi-group view which contains the effect of the rough driver. This allows for 
semi-linear perturbations, taking a mild solution point of view, with the appealing feature that we can deal with semi-linear rough PDE's essentially by semi-group methods, without further direct input from rough path analysis.

More specifically, we show well-posedness of the semi-linear rough PDE \eqref{u_introSemiLinear} by introducing the mild formulation
\begin{equation} \label{u_introSemiLinearMild}
u_{t}  = P_{0t}^{\WW} u_0 + \int_0^t P_{st}^{\WW} F(u_s) ds
\end{equation}
where $P_{st}^{\WW}g$ the solution of \eqref{u_introLinear} at time $t$ when started at time $s$ in $g \in L^2(\R^d)$. Well-posedness of \eqref{u_introSemiLinearMild} is shown using a standard fix point argument. 
Due to a technical difficulty (which could be avoided by working with $p$-variation paths instead of H\"{o}lder continuous paths) we only prove that the formulation \eqref{u_introSemiLinear} implies the formulation \eqref{u_introSemiLinearMild} and {\em not} the converse implication. This however, is enough to show well-posedness of \eqref{u_introSemiLinear}, since we can use rough path continuity to show {\em existence} of a solution to \eqref{u_introSemiLinear}. Uniqueness follows immediately from the well-posedness of \eqref{u_introSemiLinearMild}.

A crucial step in our analysis underlying \eqref{DualTrick}, \eqref{u_SquaredF} and then \eqref{u_introSemiLinear} $\implies$ \eqref{u_introSemiLinearMild} relies on a product formula for rough evolutions in Banach-spaces, see Lemma \ref{prodFormula}.

\bigskip

To the best of our knowledge, semi-linear rough PDE's of the type \eqref{u_introLinear} have not been considered in the variational setting in the literature so far. In the linear case, the recent works \cite{BG} (resp. \cite{DGHT}) consider the case $L=0$ (resp. $L=\Delta $)
and $\Gamma$ of pure transport type and use an intricate doubling of the variables argument and a rough version of the Gronwall lemma, first introduced in \cite{DGHT}, to obtain energy estimates. Using these techniques, the work \cite{HH} study the forward equation in \eqref{u_introLinear} in divergence form under optimal conditions on the coefficients in the drift term.

\section{Notation and definitions}

For $T > 0$ we define $\Delta([0,T]) := \{ (s,t) \in [0,T]^2 :  s <t \}$. For a Banach space $E$, a mapping $g : \Delta([0,T]) \rightarrow E$ will be said to be $\alpha$-H\"{o}lder continuous provided
$$
\|g \|_{\alpha} := \sup_{ (s,t) \in \Delta([0,T])  } \frac{| g_{st} |}{|t-s|^{\alpha}}  < \infty.
$$
We denote by $C^{\alpha}_2([0,T];E)$ the space of all $\alpha$-H\"{o}lder continuous functions equipped with the above semi-norm. We let $C^{\alpha}([0,T];E)$ the set of all $f :[0,T]\rightarrow E$ such that $\delta f \in C^{\alpha}_2([0,T];E)$ where we have defined the increment $\delta f_{st} : = f_t - f_s$. The second order increment for a mapping $g : \Delta([0,T]) \rightarrow E$ is by somewhat abusive notation defined by $\delta g_{s \theta t} := g_{s  t} - g_{ \theta t} - g_{s \theta }$.

\bigskip

We shall work with the usual Sobolev spaces $W^{n,p}(\R^d)$ with norm denoted $\|\cdot \|_{n,p}$, and for simplicity we denote by $H^n := W^{n,2}(\R^d)$ with norm $\| \cdot \|_n := \|\cdot \|_{n,2}$. 
For smooth and compactly supported functions $f$ and $g$ on $\R^d$, denote by $(f,g) = \int_{\R^d} f(x) g(x) dx$ and by the same bracket the extension of the bi-linear mapping
$$
(\cdot, \cdot ) : (W^{n,p}(\R^d))^* \times W^{n,p}(\R^d) \rightarrow \R.
$$
Moreover, when $q$ is such that $q^{-1} + p^{-1} = 1$, we write $W^{-n,q}(\R^d) := (W^{n,p}(\R^d))^*$. 

\bigskip

We consider the following second order operator
\begin{equation} \label{SecondOrder}
L\phi(x) = \frac{1}{2} \sigma_{i,k}(x) \sigma_{j,k}(x) \partial_i \partial_j \phi(x) + b_j(x) \partial_j \phi(x) + c(x) \phi(x)
\end{equation}
and the first order operator 
\begin{equation} \label{FirstOrder}
\Gamma^j \phi(x) = \beta_j^n(x) \partial_n \phi(x) + \gamma_j(x) \phi(x) . 
\end{equation}
Here, and for the rest of the paper we use the convention of summation over repeated indices.

The formal adjoints of these operators are given by
$$
L^*\phi(x) = \frac{1}{2} \sigma_{i,k}(x) \sigma_{j,k}(x) \partial_i \partial_j \phi(x) + \tilde{b}_j(x) \partial_j \phi(x) + \tilde{c}(x) \phi(x)
$$
and
$$
\Gamma^{j,*} \phi(x) = - \beta_j^n(x) \partial_n \phi(x) + \tilde{\gamma}_j(x) \phi(x) ,
$$
where
\begin{equation} \label{dualNotation}
\begin{array}{ll}
\tilde{b}_j(x)  = \partial_i ( \sigma_{i,k}(x) \sigma_{j,k}(x) ) - b_j(x),  &  \tilde{c}(x) = \frac{1}{2} \partial_i \partial_j ( \sigma_{i,k}(x) \sigma_{j,k}(x)) - \partial_j b_j(x) + c(x) \\
\tilde{\gamma}_j(x) = \gamma_k(x) - \partial_n \beta_j^n(x) . & \\
\end{array}
\end{equation}

\bigskip

Given a smooth $e$-dimensional path $W = (W^1, \dots W^e)$, we can define
\begin{equation} \label{IteratedIntegral}
\W^{i,j}_{st} := \int_s^t \delta W^i_{s r} \dot{W}^j_r dr .
\end{equation}
In the case of a irregular path of, e.g. a sample path of the Brownian motion, the above definition does not make sense, since $W$ is not differentiable but only $\alpha$-H\"{o}lder continuous for $\alpha$ arbitrarily close to $1/2$. In that case, however, we could choose to define the integration $\int W^i dW^j$ as e.g. as an It\^{o} integral or a Stratonovich integral. One can then show that for almost all sample paths of the Brownian motion, $\W \in C^{2 \alpha}_2([0,T]; \R^{e \times e})$ and we have the so-called Chen's relation
\begin{equation} \label{ClassicalChensRelation}
\delta \W_{s \theta t}^{i,j} := \W_{s  t}^{i,j} - \W_{ \theta t}^{i,j} - \W_{s \theta }^{i,j} = \delta W_{s \theta}^i \delta W_{\theta t}^j . 
\end{equation}

Motivated by the above, we will say that $\WW = ( W , \W) \in C^{\alpha}([0,T] ;\R^e) \times C^{ 2 \alpha}_2([0,T]; \R^{e \times e})$ for $\alpha \in (\frac{1}{3}, \frac{1}{2})$ is a \emph{rough path} provided \eqref{ClassicalChensRelation} holds. Denote $\mathscr{C}^{\alpha}([0,T];\R^e)$ the set of all rough paths and by $\|\WW\|_{\alpha} := \|W\|_{\alpha} + \sqrt{ \| \W \|_{2 \alpha} }$ the induced metric. We shall say that $\WW$ is a geometric rough path if there exists a sequence, $W(n)$, of smooth paths such that if $\W(n)$ defined by \eqref{IteratedIntegral} with $W(n)$ instead of $W$, we have that $\WW(n) \rightarrow \WW$ with respect to the metric $\| \cdot \|_{\alpha}$. We denote by $\mathscr{C}^{\alpha}_g$ the subset of all geometric rough paths, and notice that they satisfy the following symmetry 
$$
\W_{s  t}^{i,j} + \W_{s  t}^{j,i} = \delta W_{st}^i \delta W_{st}^j.
$$
Indeed, it is enough to notice that this is satisfied for any smooth path, and then take the limit in the rough path metric.

We shall use the notion of a controlled rough path as first introduced in \cite{Gubinelli04}.

\begin{definition}[\textbf{Controlled path}] \label{def:ControlledPath}
A pair of functions 
$$
Y : [0,T]  \rightarrow  E \hspace{.5cm} and \hspace{.5cm} Y' : [0,T]  \rightarrow  E^{e}
$$ 
is said to be controlled by $W$ in $E$ provided
$$
|  \delta  Y_{st}  -   (Y_s')^{i}  W_{st}^i | \lesssim |t-s|^{2 \alpha} \hspace{.5cm} and \hspace{.5cm} | \delta Y'_{st}| \lesssim |t-s|^{\alpha} . 
$$
We denote by $\| (Y,Y') \|_{\alpha, W;E}$ the infimum over all constants such that the above analytic bounds hold, and by $\mathscr{D}_W^{2 \alpha}([0,T]; E)$ the linear space of all such pairs $(Y,Y')$, which we equip with the semi-norm $\| (\cdot,\cdot') \|_{\alpha, W;E}$. We shall sometimes refer to $Y'$ as the {\em Gubinelli}-derivative. 

\end{definition}

\begin{remark}
The statement "controlled by $W$ \emph{in $E$}" is not standard, but it allows us to differentiate between the strong and weak solutions, by letting $E$ be either $L^p(\R^d)$ or $W^{-3,p}(\R^d)$ respectively.
\end{remark}

The main technical tool for constructing integrals w.r.t. rough paths is the following result. For a proof, see e.g. \cite{FH14}.

\begin{lemma}[\textbf{Sewing lemma}] \label{sewingLemma}
Assume $G : \Delta([0,T]) \rightarrow E$ be such that $|\delta G_{s \theta t}| \leq K |t-s|^{\alpha}$ for some $\alpha > 1$. Then there exists a unique pair $I: [0,T] \rightarrow E$ and $I^{\natural} : \Delta([0,T]) \rightarrow E$ satisfying
$$
\delta I_{st} = G_{st} + I_{st}^{\natural}
$$
where $|I_{st}^{\natural}| \leq K C_{\alpha}|t-s|^{\alpha}$ where $C_{\alpha}$ depends only on $\alpha$. Moreover, if $|G_{st}| \lesssim |t-s|^{\beta}$ we have $|I_{st}| \lesssim |t-s|^{\beta}$.

\end{lemma}

For a controlled path we may define the rough path integral of $u$ w.r.t. $W$ as follows; define the local expansion $G \in E^e$ with components
$$
G_{st}^j :=   Y_s W_{st}^{j} + (Y_s')^{ i}  \mathbb{W}_{st}^{i,j}.
$$
Using Chen's relation \eqref{ClassicalChensRelation} we have
$$
\delta G_{s \theta t}^j = - \left(  \delta Y_{s \theta} - (Y_s')^{i}  W_{s \theta}^i \right) W_{\theta t}^j -  (\delta Y_{s \theta}')^{i}  \mathbb{W}_{\theta t}^{i,j}
$$
so that
$$
| \delta G_{s \theta t} | \leq 2  \| (Y,Y') \|_{\alpha, W;E} \|\mathbf{W}\|_{\alpha}  |t-s|^{3 \alpha} .
$$
By Lemma \ref{sewingLemma} there exists a unique path $I$ with values in $E^e$ such that 
$$
\delta I_{st} = G_{st} + I_{st}^{\natural}
$$
and we have 
$$
| I_{st}^{\natural} | \leq C_{\alpha} \| (Y,Y') \|_{\alpha, W;E} \|\mathbf{W}\|_{\alpha} |t-s|^{3 \alpha} . 
$$

\begin{definition}
We denote by $\int_0^{\cdot} (Y,Y')_r d \mathbf{W}_r$ the path $I$ obtained in the above way, called the rough path integral of $Y$ against $W$.
\end{definition}

\begin{remark} \label{linearMappingRemark}
For a continuous linear mapping $T : E \rightarrow F$ we have
$$
T( \int_0^{\cdot} (Y, Y')_r d\WW_r ) =  \int_0^{\cdot} ( TY, TY')_r d\WW_r 
$$
where 
$$
TY : [0,T]  \rightarrow  F \hspace{.5cm} and \hspace{.5cm} TY' : [0,T]  \rightarrow  F^{e}
$$ 
is defined by $(T(Y'))^{i} = T( (Y')^{i})$.

In particular, if for some Banach space $V$ 
$$
Y : [0,T]  \rightarrow  V^* \hspace{.5cm} and \hspace{.5cm} Y' : [0,T]  \rightarrow  (V^*)^{e }
$$
satisfies 
$$
|  \delta  Y_{st}(\phi)  -   (Y_s')^{i}(\phi)  W_{st}^i | \lesssim |\phi|_V |t-s|^{2 \alpha} \hspace{.5cm} and \hspace{.5cm} | \delta (Y')^{i}_{st}(\phi)| \lesssim |t-s|^{\alpha} |\phi|_V
$$
for all $ \phi \in V$ we may define the rough path integral $\int_0^{\cdot} (Y,Y')_r d\WW_r : [0,T] \rightarrow (V^*)^e$ as the unique function satisfying 
$$
| \int_s^t (Y,Y')_r d\WW_r^j (\phi)  - Y_{s}(\phi)W_{st}^j - (Y_{s}')^{j,i}(\phi) \mathbb{W}_{st}^{j,i} | \lesssim |\phi|_V |t-s|^{3 \alpha}
$$
\end{remark}

The following lemma is the main technical tool of the paper.
It replaces the tensorization argument in \cite{BG}, \cite{DGHT} and \cite{HH}.

\begin{lemma} \label{prodFormula}
Assume $u : [0,T] \rightarrow E^*$ satisfies 
$$
u_t = u_0 + \int_0^t  A_r dr + \int_0^t (B_r,B'_r) d\WW_r
$$
for some $A \in L^2([0,T]; E^*)$ and $(B, B') = (B^j, (B')^j)$ controlled by $W$ in $E^*$, and we have set
$$
\int_0^t (B_r,B'_r) d\WW_r := \int_0^t ((B_r)^j,(B'_r)^j) d\WW_r^j.
$$

Moreover, assume $f: [0,T] \rightarrow E$ satisfies
$$
f_t = f_0 + \int_0^t K_r dr + \int_0^t (N_r,N'_r)d\WW_r
$$
for some $K \in L^2([0,T]; E)$ and $(N,N') = (N^j, (N')^j)$ controlled by $W$ in $E$. 
In addition, we assume that $u$ (respectively $f$) is controlled by $W$ in $E^*$ (respectively $E$). 

Then, if $\WW$ is a geometric rough path we have
\begin{align*}
u_t(f_t) = u_s(f_s) + \int_s^t A_r(f_r) + u_r(K_r) dr  + \int_s^t (M_r,M_r') d\WW_r
\end{align*}
where 
$$
M_t^j = (B^j_t,f_t) + (u_t,N^j_t) \hspace{1cm} (M_t')^{j,i} = ((B_t')^{j,i},f_t) + 2 (B_t^j,N_t^i)  + (u_t, (N_t')^{j,i}) . 
$$
\end{lemma}

%
%\begin{remark}
%Although it will always be true in our setting, the additional assumption above that $u$ is controlled by $W$ in $E^*$ (and similarly for $f$) may seem redundant. However, for the equation \eqref{u_introSemiLinear} we shall see that we have $B^j = \Gamma^j u$ and $(B')^{j,i} = \Gamma^i \Gamma^j u$ and the full Taylor expansion of the integral is
%$$
%\int_s^t (B_r,B'_r) d\WW_r =  \Gamma^ju_s W^j_{st} + \Gamma^i \Gamma^j u_s \mathbb{W}^{j,i}_{st} + u^{\natural}_{st} .
%$$
%Supposing also $f$ solves a similar equation, we get
%$$
%\int_s^t (N_r,N'_r) d\WW_r =  \bar{\Gamma^j} f_s W^j_{st} +  \bar{\Gamma^i} \bar{\Gamma^j}  f_s \mathbb{W}^{j,i}_{st} + f^{\natural}_{st},
%$$
%for a possibly different first order differential operator $\bar{\Gamma}^i$.  
%Brute force computations for the product would require high regularity on $f$ (i.e. $f \in W^{5,p}(\R^d)$), which again would necessitate smoother coefficients.
%
%
%\end{remark}

\begin{proof}
Assume for simplicity that $A = K = 0$. By definition of $u$, it is the unique path $[0,T] \rightarrow E^*$ such that 
$$
u^{\natural}_{st}(\phi) := \delta u_{st}(\phi) - \left[ B^j_s(\phi) W^j_{st} + (B_s')^{j,i}(\phi) \mathbb{W}^{i,j}_{st} \right] 
$$
satisfies $|u^{\natural}_{st}(\phi)| \lesssim |t-s|^{3 \alpha} |\phi| $. Moreover, by assumption we have
$$
u^{\flat}_{st} := \delta u_{st} - B_s^j W_{st}^j \in C^{2 \alpha}_2([0,T]; E^*)
$$

Similarly $f$ is the unique path $f : [0,T] \rightarrow E $ such that 
$$
f^{\natural}_{st} := \delta f_{st} - \left[ N^j_s W^j_{st} + (N_s')^{j,i} \mathbb{W}^{i,j}_{st} \right] 
$$
satisfies $|f_{st}^{\natural}|  \lesssim |t-s|^{3 \alpha}$, and by assumption we have
$$
f^{\flat}_{st} := \delta f_{st} - N_s^j W_{st}^j \in C^{2 \alpha}_2([0,T]; E). 
$$
Algebraic manipulations give
\begin{align*}
\delta u(f)_{st}  & = \delta u_{st} (f_s) + u_s(\delta f_{st}) + \delta u_{st}(\delta f_{st}) \\
 & = B^j_s(f_s) W^j_{st}   + (B_s')^{j,i}(f_s) \mathbb{W}^{i,j}_{st}  + u_{st}^{\natural}(f_s) + u_s(N^j_s) W^j_{st}  + u_s((N_s')^{j,i}) \mathbb{W}^{i,j}_{st} + u_s(f_{st}^{\natural})\\ 
  & + (u_{st}^{\flat} + B_s^j W_{st}^j)( f_{st}^{\flat} + N_s^i W_{st}^i) \\
 &  = \left[ B^j_s(f_s) + u_s(N^j_s) \right] W^j_{st} + \left[  (B_s')^{j,i}(f_s) + u_s((N_s')^{j,i}) + 2  B^j_s(N^i_s) \right] \mathbb{W}_{st}^{i,j} + u(f)^{\natural}_{st}
\end{align*}
where we have used that $\WW$ is geometric and we have defined
\begin{align*}
u(f)^{\natural}_{st} &  := u_s(f_{st}^{\natural}) + u_{st}^{\natural}(f_s)  + B_s^j(f_{st}^{\flat}) W_{st}^j + u_{st}^{\flat}(N_s^j) W_{st}^j 
\end{align*}
It is easy to see that we have $|u(f)^{\natural}_{st}| \lesssim |t-s|^{3 \alpha}$. Since $3 \alpha > 1$, the result follows from the uniqueness of Lemma \ref{sewingLemma}.
\end{proof}

With a definition of the rough integral at hand we can go on to define the notion of a solution to our main equations.

\begin{definition}[\textbf{Backward RPDE solution}]
\label{def:strongForwardSolutionRough copy(2)} Given an $\alpha $-H\"{o}lder
rough path $\mathbf{W}=(W,\mathbb{W})$, $\alpha \in (1/3,1/2]$
we say that $u\in C([0,T],W^{3,p}(\R^d))$ is a regular backward solution
to 
\begin{equation*}
\begin{cases}
-du_{t} & =Lu_{t}dt+\Gamma^i u_{t}d\mathbf{W}^i_{t} \\ 
u_{T} & \in W^{3,p}(\R^d),%
\end{cases}%
\end{equation*}%
provided $(\Gamma^i u, -\Gamma^i \Gamma^j u )$ is controlled by $W$ in $L^p(\R^d)$ 
and if the following holds as equality in $L^{p}(\R^d)$ 
\begin{equation*}
u_{t}=u_{T}+\int_{t}^{T}Lu_{s}ds+\int_{t}^{T}\Gamma^i u_{s}d\mathbf{W}^i_{s}.
\end{equation*}%

We say that $u\in C([0,T],L^{p}(\R^d))$ is an analytically weak solution if $(\Gamma^i u,  - \Gamma^i \Gamma^j u )$ is controlled by $W$ in $W^{-3,p}(\R^d)$ and the following equality holds in $W^{-3,p}(\R^d)$
\begin{equation*}
u_{t}=u_{T}+\int_{t}^{T}Lu_{s}ds+\int_{t}^{T}\Gamma^i u_{s}d\mathbf{W}^i_{s}.
\end{equation*}
Equivalently, for all $\phi \in W^{3,q}(\R^d)$
\begin{equation*}
( u_{t},\phi ) =( u_{T},\phi
) +\int_{t}^{T}( u_{s},L^{\ast }\phi )
ds+\int_{t}^{T}( u_{s},\Gamma ^{i,\ast }\phi ) d%
\mathbf{W}^i_{s}.
\end{equation*}
\end{definition}

\begin{definition}[\textbf{Forward RPDE solution}]
\label{def:strongForwardSolutionRough} Given an $\alpha $-H\"{o}lder rough
path $\mathbf{W}=(W,\mathbb{W})$, $\alpha \in (1/3,1/2]$, we say that $v \in C([0,T]; W^{3,p}(\R^d))$ is a regular forward solution to
\begin{equation*}
\begin{cases}
dv_{t} & =L^{\ast }v_{t}dt+\Gamma ^{i,\ast }v_{t}d\mathbf{W}^i_{t} \\ 
v_{0} & \in W^{3,p}(\R^d),%
\end{cases}%
\end{equation*}%
provided $(\Gamma^i v, \Gamma^i \Gamma^jv )$ is controlled by $W$ in $L^p(\R^d)$ and the following holds as equality in $L^{p}(\R^d)$ 
\begin{equation*}
v_{t}=v_{0}+\int_{0}^{t}L^{\ast }v_{s}ds+\int_{0}^{t}\Gamma ^{i,\ast }v_{s}d%
\mathbf{W}^i_{s}.
\end{equation*}%
Equivalently, for all $\varphi \in L^{q}(\R^d),$%
\begin{equation*}
( v_{t},\varphi ) =( v_{0},\varphi
) +\int_{0}^{t}( L^{\ast }v_{s},\varphi )
ds+\int_{0}^{t}( \Gamma ^{i, \ast }v_{s},\varphi ) d%
\mathbf{W}_{s}^i.
\end{equation*}%

We say that $v\in C([0,T],L^{p}(\R^d))$ is an analytically weak solution if $(\Gamma^i v, \Gamma^i \Gamma^jv)$ is controlled by $W$ in $W^{-3,p}(\R^d)$ and the following equality holds in $W^{-3,p}(\R^d)$
\begin{equation*}
v_{t}=v_{0}+\int_{0}^{t}L^{\ast }v_{s}ds+\int_{0}^{t}\Gamma ^{i,\ast }v_{s}d%
\mathbf{W}^i_{s}.
\end{equation*}%
Equivalently, for all $\phi \in W^{3,q}(\R^d)$ 
\begin{equation*}
( v_{t},\phi ) =( v_{0},\phi
) +\int_{0}^{t}( v_{s},L\phi )
ds+\int_{0}^{t}( v_{s},\Gamma^i \phi ) d\mathbf{W}^i_{s}.
\end{equation*}
\end{definition}

\begin{remark}
Given a backward RPDE solution, driven by $\left( \mathbf{W}_{t}\right) $
with terminal data $u_{T}$, it is easy to see that $v_{t}:=u_{T-t}$ solves a
forward RPDE solution driven by $\left( \mathbf{W}_{T-t}\right) 
$ and initial data $v_{0}=u_{T}$.
\end{remark}

\section{Well-posedness of linear equations}

We denote by $X$ the solution of 
\begin{equation} \label{SDE}
dX_t = \sigma(X_t) dB_t + b(X_t) dt + \beta(X_t) d\WW_t.
\end{equation}
The main objective of this section is to prove that 
\begin{equation} \label{backwardFK}
u(t,x) = E^{(t,x)} \left[ g(X_T) \exp \left\{ \int_t^T c(X_r) dr + \int_t^T \gamma(X_s) d \WW_s \right\} \right]
\end{equation}
is a solution to the backward equation
$$
- du_t = L u_t dt + \Gamma^i u_t d \WW^i_t, \quad u_T = g .
$$

We will show that when $g \in L^p(\R^d)$ (respectively $g \in W^{3,p}(\R^d)$) the above expressions yield a weak solution (respectively regular solution). When $W$ is a smooth path, this is already well known, and we will prove the result by showing that $(\Gamma^i u, - \Gamma^j \Gamma^i u)$ is controlled by $W$ in $W^{-3,p}(\R^d)$ (respectively $L^p(\R^d)$) as well as using the rough path stability of the controlled spaces.

A step towards this goal is to consider the diffusion $X$ as a solution to the rough differential equation 
\begin{equation} \label{roughSDE}
dX_t = b(X_t) dt + V_j(X_t) d \mathbf{Z}_t^j
\end{equation}
where we have defined the $d_B +e$-dimensional rough path $\mathbf{Z}  = (Z, \mathbb{Z})$  where
$$
Z_t = \left(\begin{array}{l}
B_t \\
W_t \\
\end{array} \right)
\hspace{.5cm} and \hspace{.5cm} 
\mathbb{Z}_{st} = \left(\begin{array}{cc}
\mathbb{B}_t & \int_s^t B_{sr} dW_r \\
\int_s^t W_{sr} dB_r & \mathbb{W}_{st} \\
\end{array} \right)
$$
and $V_i = \sigma_i$ for $i=1, \dots, d_B$ and $V_i =  \beta_i$ for $i = d_B+1, \dots d + e$. Above, the term $\int_s^t W_{sr}^i dB^j_r$ is the Wiener integral of the deterministic function $W_{s \cdot}^i$, and we define $\int_s^t B_{sr}^i dW^j_r := B_{st}^i W_{st}^j - \int_s^t W_{sr}^j dB_r^i$. For more details, see \cite{DOR}. We will denote by $\Phi$ the flow generated by \eqref{roughSDE}.

\subsection{Weak solutions}

\begin{theorem} \label{existenceThm}
Assume $\sigma_{i,k}, \beta_{j} \in C_b^3(\R^d)$, $b_j, \gamma_j, c \in C_b^1(\R^d)$. Given $g \in L^{p}(\R^d)$, the Feynman-Kac formula \eqref{backwardFK} yields an analytically weak
backward RPDE solution $u$. 
\end{theorem}

\begin{proof}
For simplicity we assume $b = c = \gamma = 0$.

We start by showing that $u$ is a well defined element of $L^p(\R^d)$. Define the random variable $J := \sup_{x} |det( \nabla  \Phi_{t,T}^{-1}(x) ) |$. From Lemma \ref{roughLiouvilleLemma} and Proposition \ref{summary} we know that $E[ J ] < \infty$.
We write
\begin{align*}
\int |u_t(x)|^p dx & = \int \left| E[ g(\Phi_{t,T}(x) )] \right|^p dx  = \int \left| E \left[ \frac{g(\Phi_{t,T}(x) ) J^{1/p}}{J^{1/p}} \right] \right|^p dx \leq E [J] \int  E\left[ \frac{|g(\Phi_{t,T}(x))|^p }{J} \right]  dx\\
 &  = E [J]  E\left[ \int   \frac{|g(\Phi_{t,T}(x))|^p }{J}  dx \right]  =  E [J]  E\left[ \int   \frac{|g(y)|^p }{J}  |det( \nabla  \Phi_{t,T}^{-1}(y) ) |dy \right] \leq E [J]  \int  |g(y)|^p dy ,
\end{align*}
where we have used H\"{o}lder's inequality and $J^{-1}|det( \nabla  \Phi_{t,T}^{-1}(y) ) | \leq 1$. The mapping $g \mapsto u$ is linear and continuous on $L^p(\R^d)$, giving
\begin{equation} \label{InitialLinear}
\sup_{t \in [0,T]} \|u_t \|_{p,0} \leq C \|g\|_{0,p} .
\end{equation}

We now show that $(\Gamma^j u, - \Gamma^j \Gamma^i u)$ is controlled by $W$ in $W^{-3,p}(\R^d)$. Fix $\phi \in W^{3,q}(\R^d)$ and notice that $\Gamma^{j,*}\phi \in W^{2,q}(\R^d)$. By Lemma \ref{LemmaL2Composition} we see that
$$
\int g( \Phi_{\cdot,T}(x)) \Gamma^{j,*} \phi(x) dx = \int g(y) \Gamma^{j,*}\phi( \Phi_{\cdot,T}^{-1}(y) ) det( \nabla \Phi_{\cdot,T}^{-1}(y) ) dy
$$
and
$$
\int g(y) (\Gamma^{i,*} \Gamma^{j,*} \phi)( \Phi_{\cdot,T}^{-1}(y) ) det( \nabla \Phi_{\cdot,T}^{-1}(y) ) dy
$$
is $P$-a.s. controlled by $Z = (B,W)$ in $\R$ and we have the estimate
\begin{align*}
\big\| \big( \int g(y) \Gamma^{j,*}\phi( \Phi_{\cdot,T}^{-1}(y) ) & det( \nabla \Phi_{\cdot,T}^{-1}(y) ) dy , \int g(y) (div(V_i \Gamma^{j,*} \phi)( \Phi_{\cdot,T}^{-1}(y) ) det( \nabla \Phi_{\cdot,T}^{-1}(y) ) dy \big) \big\|_{\alpha, Z; \R} \\
 & \leq C \|g \|_{0,p} \| \phi \|_{3,q} \exp \{ C N_{  [0,T]}(\mathbf{Z}) \} (1 + \|\mathbf{Z} \|_{\alpha})^k .
\end{align*}
Written explicitly, we have

\begin{align*}
\big| \delta \big( \int g(y) \Gamma^{j,*}\phi( \Phi_{\cdot,T}^{-1}(y) ) & det( \nabla \Phi_{\cdot,T}^{-1}(y) ) dy \big)_{st}   - \int g(y) ( div( \sigma_i \Gamma^{j,*} \phi)( \Phi_{t,T}^{-1}(y) ) det( \nabla \Phi_{t,T}^{-1}(y) ) dy B_{st}^i  \\
  & - \int g(y) ( div( \beta_i \Gamma^{j,*} \phi)( \Phi_{t,T}^{-1}(y) ) det( \nabla \Phi_{t,T}^{-1}(y) ) dy W_{st}^i \big| \\
  & \leq C \|g \|_{0,p} \| \phi \|_{3,q} \exp \{ C N_{  [0,T]}(\mathbf{Z}) \} (1 + \|\mathbf{Z} \|_{\alpha})^k |t-s|^{2 \alpha}.
\end{align*}
Using the above, independence of Brownian increments and the fact that $\Gamma^{i,*} \psi = - div( \beta_i \psi)$ we get

\begin{align*}
\big| \delta \big( E\big[ \int g(y)  & \Gamma^{j,*}\phi( \Phi_{\cdot,T}^{-1}(y) ) det( \nabla \Phi_{\cdot,T}^{-1}(y) )  dy \big] \big)_{st} +  \int g(y) E\left[( \Gamma^{i,*} \Gamma^{j,*} \phi)( \Phi_{t,T}^{-1}(y) ) det( \nabla \Phi_{\cdot,T}^{-1}(y) ) \right]dy W_{st}^i \big|  \\
 &  \leq C \|g \|_{0,p} \| \phi \|_{3,q} E \big[ \exp \{ C N_{  [0,T]}(\mathbf{Z}) \} (1 + \|\mathbf{Z} \|_{\alpha})^k \big] |t-s|^{2 \alpha},
\end{align*}
 which proves that $(\Gamma^i u,  - \Gamma^{j} \Gamma^{i}u)$ is controlled by $W$ in $W^{-3,p}(\R^d)$.

\bigskip
 
To see that $u$ is an analytically weak solution we argue by rough path continuity. In fact, for $\WW$ smooth it is well known that $u$ is an analytically weak solution of 
$$
- \partial_t u = L u_t + \Gamma^j u_t \dot{W}_t .
$$
The continuity $\WW \mapsto u$ from $\mathscr{C}^{\alpha}_g$ to $W^{-3,p}(\R^d)$ equipped with the weak*- topology is clear; by density it suffices to take $\phi \in C^{\infty}_c(\R^d)$ and show the continuity $\WW \mapsto (u_t, \phi)$. The latter is equal to
\begin{align*}
\int E[  g( \Phi_{t,T}(x)) ] \phi(x) dx &  = E \left[ \int g( \Phi_{t,T}(x))  \phi(x) dx \right] =  \int g(y)  E \left[  \phi( \Phi_{t,T}^{-1}(y) ) det( \nabla \Phi_{t,T}^{-1}(y) ) \right] dy	 
\end{align*}
which is continuous w.r.t. $\WW$.
We can then take the limit in every term in the equation, using rough path stability to see that $u$ indeed satisfies the equation.
\end{proof}

\subsection{Regular solutions}

\begin{theorem} \label{existenceThmRegular}
Assume $\sigma_{i,k}, \beta_{j}, \gamma_j \in C_b^6(\R^d)$, $b_j, c \in C_b^4(\R^d)$. Given $g \in W^{3,p}(\R^d)$, the Feynman-Kac formula \eqref{backwardFK} yields a regular backward RPDE solution.

\end{theorem}

\begin{proof}
Assume for simplicity that $\gamma_j = b_j= c =0$. We start by showing that $u$ is a well defined element of $W^{3,p}(\R^d)$. To see that $u \in W^{1,p}(\R^d)$ write $\nabla u_t(x) = E[\nabla g( \Phi_{t,T}(x)) \nabla  \Phi_{t,T}(x)]$. With the notation of the proof of Theorem \ref{existenceThm} we write
\begin{align*}
\int_{\R^d} |\nabla u_t(x) |^p dx & = \int_{\R^d} \left| E \left[ \frac{\nabla g( \Phi_{t,T}(x)) \nabla  \Phi_{t,T}(x) J^{1/p}}{J^{1/p}}  \right] \right|^p  dx \\ 
 &  \leq \sup_{x \in \R^d} E[ |\nabla  \Phi_{t,T}(x)|^p  J]  E \left[ \int_{\R^d} \frac{|\nabla g( \Phi_{t,T}(x)) |^p }{J} dx \right] .
\end{align*}
The latter factor can be bounded in the same way as in the proof of Theorem \ref{existenceThm}. From Proposition \ref{summary2} and Lemma \ref{roughLiouvilleLemma} the first factor is also bounded. 

By the assumptions on  $\sigma_{i,k}, \beta_{j}, \gamma_j,b_j, c $ the flow map $x \mapsto \Phi_{t,T}(x)$ is $C^3_b(\R^d)$, and by iterating the above we can show that $u_t \in W^{3,p}(\R^d)$. Details are left to the reader.

To see that $u$ is controlled by $W$ in $L^p(\R^d)$ we notice that we have
$$
(\Phi_{\cdot,T},  - \nabla \Phi_{\cdot,T} V ) \in \mathscr{D}_{Z}^{2 \alpha}([0,T]; L^{\infty}(\R^d))
$$
and
$$
(\nabla \Phi_{\cdot,T},  - \nabla^2 \Phi_{\cdot,T} V - \nabla \Phi_{\cdot,T} \nabla V  ) \in \mathscr{D}_{Z}^{2 \alpha}([0,T]; L^{\infty}(\R^d))
$$
see \cite[Lemma 32]{DFS}. Moreover, 
$$
\nabla u_t(x) = E[ \nabla g(\Phi_{t,T}(x)) \nabla \Phi_{t,T}(x) ]
$$
so that by Lemma \ref{L2Multiplier} we have
$$
(\nabla g(\Phi_{\cdot,T}) \nabla \Phi_{\cdot,T}, - \nabla\left[  \nabla g(\Phi_{\cdot,T}) \nabla \Phi_{\cdot,T} V \right]   ) \in \mathscr{D}_{Z}^{2 \alpha}([0,T]; L^{p}(\R^d)) .
$$
Written explicitly,
\begin{align*}
\big\| \delta \left( \nabla g(\Phi_{\cdot,T}) \nabla \Phi_{\cdot,T} \right)_{st}  & + \nabla\left[  \nabla g(\Phi_{s,T}) \nabla \Phi_{s,T} V_j \right] Z_{st}^j \big\|_{L^p} \\
 &   \leq C \exp\{ C N_{[0,T]}(\ZZ) \}(1 + \| \ZZ \|_{\alpha})^k \|g\|_{3,p} |t-s|^{2 \alpha}.
\end{align*}
Integrated against $ \beta_i \phi \in L^q(\R^d)$ gives
\begin{align*}
\big| \int  \beta_i (x) \delta \left( \nabla g(\Phi_{\cdot,T})(x) \nabla \Phi_{\cdot,T}(x) \right)_{st} \phi(x) dx   & + \int \beta_i (x) \nabla\left[  \nabla g(\Phi_{s,T}(x)) \nabla \Phi_{s,T}(x) \sigma_j(x) \right] B_{st}^j  \\
 & + \int \beta_i (x) \nabla\left[  \nabla g(\Phi_{s,T}(x)) \nabla \Phi_{s,T}(x) \beta_j(x) \right] W^j_{st} \big| \\
 &  \leq C \exp\{ C N_{[0,T]}(\ZZ) \} (1 + \| \ZZ \|_{\alpha})^k \|g\|_{p,3} \| \phi \|_{0,q} |t-s|^{2 \alpha} .
\end{align*}

Using the above, independence of Brownian increments and the fact that $\Gamma^i\psi = \beta_i \nabla \psi$ we get
\begin{align*}
\big| \int \delta (\Gamma^i u)_{st}(x) \phi(x) dx   &  + \int \Gamma^j ( \Gamma^i u_s)(x) \phi(x) dx  W^j_{st} \big| \\
 &  \leq C E\left[  \exp\{ C N_{[0,T]}(\ZZ) \} (1 + \| \ZZ \|_{\alpha})^k  \right] \|g\|_{3,p} \| \phi \|_{0,q} |t-s|^{2 \alpha} .
\end{align*}
The proof that $u$ in fact satisfies the equation is similar as in the proof of Theorem \ref{existenceThm}.
\end{proof}

We end this section with an ad-hoc result that will be needed to use Lemma \ref{prodFormula}. 
\begin{proposition} \label{existenceSpaceSmooth}
Assume $\sigma_{i,k}, \beta_{j}, \gamma_j \in C_b^6(\R^d)$, $b_j, c \in C_b^4(R^d)$. Given $g \in W^{6,p}(\R^d)$, the Feynman-Kac formula \eqref{backwardFK} yields a regular backward RPDE solution with values in $W^{3,p}(\R^d)$, the rough integral is an element of $W^{3,p}(\R^d)$ and the solution is controlled by $W$ in $W^{3,p}(\R^d)$.
\end{proposition}

\begin{proof}
The proof is similar to the proof of Theorem \ref{existenceThmRegular}, we only need to check that $(\Gamma^iu, - \Gamma^i \Gamma^j u)$ is controlled by $W$ in $W^{3,p}(\R^d)$. The result follows from the same argument as in Theorem \ref{existenceThmRegular} coupled with Lemma \ref{L2CompositionH3}.
\end{proof}

\subsection{Uniqueness}

In this section we prove the uniqueness of the solutions for a certain class of coefficients. The proof is based on a duality trick; existence of the solution to the backward problem plus existence of a solution to the forward problem couple with the product formula, Lemma \ref{prodFormula}, gives uniqueness via a standard trick. 

Given $g=u_{T}\in L^{p}(\R^d)$, we show uniqueness of weak solutions to
the backward RPDE 
\begin{equation*}
( u_{t},\phi ) =( g,\phi
) +\int_{t}^{T}( u_{r},L^{\ast }\phi )
dr+\int_{t}^{T}\,( u_{r},\Gamma ^{j,\ast }\phi ) d\mathbf{W}^j_{r},\qquad 0\leq t\leq T.
\end{equation*}
for $\phi \in W^{3,q}(\R^d)$.

\begin{theorem} \label{Uniqueness}
Assume $\sigma_{i,k}, \beta_{j}, \gamma_j \in C_b^6(\R^d)$, $b_j, c \in C_b^4(\R^d)$. Given $u_{T}\in L^{p}(\R^d)$, there exists a unique analytically weak backward RPDE
solution $u$. A similar result holds in the forward case.
\end{theorem}

\begin{proof}
Since the equation is linear, it is enough to prove that the only solution to 
$$
-du_t = Lu_t dt + \Gamma^{j} u_t d\WW^j_t \hspace{.5cm} u_T = 0 ,
$$
is the trivial solution $u = 0$. For simplicity we show that $u_0 = 0$.

For $\phi \in W^{6,q}(\R^d)$, denote by $v$ the regular forward solution to  
$$
dv_t = L^*v_t dt + \Gamma^{j,*} v_t d\WW^j_t \hspace{.5cm} v_0 = \varphi ,
$$
as constructed in Proposition \ref{existenceSpaceSmooth}. It is clear that $u$ is controlled by $W$ in $W^{-2,p}(\R^d)$, and consequently also in $W^{-3,p}(\R^d)$.

From Lemma \ref{prodFormula} we get that 
\begin{align*}
u_T(v_T) &  = u_0(\varphi) + \int_0^T    (u_t , L^*v_t)  - (Lu_t, v_t) dt + \int_0^T (M_r, M_r') d\WW_r 
\end{align*}
where 
$$
M_t^j = - (\Gamma^j u_t, v_t) + (u_t, \Gamma^{j,*} v_t)  = 0 \hspace{0.5cm} (M'_t)^{j,i} = (\Gamma^j \Gamma^i u_t, v_t) - 2(\Gamma^i u_t, \Gamma^{j,*} v_t) + (u_t, \Gamma^{i,*}\Gamma^{j,*} v_t) = 0.
$$
The bounded variation term is obviously equal to $0$, which gives
$$
0 = u_0(\varphi)
$$
for all $\varphi \in W^{6,q}(\R^d)$ which implies that $u_0 = 0$ in $L^p(\R^d)$.
\end{proof}

\section{Energy estimates}
In this section we use a new method, first introduced in \cite{HNS}, to find energy estimates. The method relies on finding a set of suitable space-time test functions that equilibrate the energy of the noise in the system. We assume $u$ is a weak solution to the forward equation 
$$
du_t = Lu_t + \Gamma^i u_t d\WW^i , \quad u_0 = g \in L^2(\R^d) .
$$

The main result of this section is the following.

\begin{theorem} \label{thmEnergyEstimate}
Suppose $\sigma_{i,k}, \gamma_j, \beta_j^n \in C^6_b(\R^d)$ and $b_j, c \in C_b^4(\R^d)$. Moreover, assume the following non degeneracy condition
\begin{equation} \label{diffusionNonDegeneracy}
\lambda | \xi |^2 \leq  \sigma_{i,k} \sigma_{j,k}  \xi_j \xi_i
\end{equation}
for some constant $\lambda > 0$.

Then $u \in C([0,T]; L^2(\R^d)) \cap L^2([0,T]; H^1)$, and the following energy inequality holds
\begin{equation} \label{energyEstimate}
\sup_{t \in [0,T]}  \|u_t\|_{0}^2 + \int_0^T \| \nabla u_r \|_{0}^2 dr \leq C \|g\|_{0}^2 ,
\end{equation}
where $C$ can be chosen uniformly in bounded sets of $\WW$ and depends on $\lambda$ as well as $\sigma_{i,k}, \gamma_j, \beta_j^n \in C^6_b(\R^d)$ and $b_j, c \in C_b^4(\R^d)$.
\end{theorem}

The strategy of the proof is as follows. First we prove Theorem \ref{thmEnergyEstimate} for $g$ smooth. Then, since the solution is actually a regular solution, we may use Lemma \ref{prodFormula} to multiply the solution by itself and obtain an equation for $u^2$. We then again use the product formula, Lemma \ref{prodFormula}, and a solution to a backward problem defined on $W^{3, \infty}(\R^d)$ to transform the equation for $u^2$ into an expression without a rough path integral term. This method replaces the so-called "Rough Gronwall" in \cite{DGHT}, \cite{HH}.

Since the estimate in Theorem \ref{thmEnergyEstimate} is uniform in $\|g\|_{0}$ we use the stability $g \mapsto u$ in $L^2(\R^d)$ to extend to any $g \in L^2(\R^d)$.

\begin{lemma} \label{SquaredEquation}
Suppose $g \in C^{\infty}(\R^d) \cap L^2(\R^d)$ and $\sigma_{i,k}, \gamma_j, \beta_j^n \in C^6_b(\R^d)$ and $b_j, c \in C_b^4(\R^d)$. Then $u^2$ satisfies 
$$
d u^2_t = 2 u_t Lu_t dt + 2 u_t \Gamma^i u_t d\WW^i_t
$$
on $(W^{3, \infty}(\R^d))^*$, i.e. for any $\phi \in W^{3, \infty}(\R^d)$ it holds that
$$
(u_t^2,\phi)  = (u_0^2,\phi) + 2  \int_0^t (Lu_r, \phi u_r)  dr  + 2 \int_0^t (M_r,M_r')(\phi) d \WW_r
$$
where 
$$
M_t^j(\phi) = (\Gamma^j u_t, \phi u_t)  \hspace{1cm} (M_t')^{j,i}(\phi) =  (\Gamma^i \Gamma^j  u_t,\phi u_t) + (\Gamma^i u_t, \phi \Gamma^j u_t) 	.
$$
\end{lemma}

\begin{proof}
Since $g$ is smooth we know that 
$$
u_t = g + \int_0^t Lu_r dr + \int_0^t\Gamma^j u_r d \WW^j_r
$$
holds on $H^3$. Since $W^{3, \infty}(\R^d)$ is a multiplier on $H^3$ we have from Remark \ref{linearMappingRemark} that
$$
\phi u_t = \phi g + \int_0^t \phi  Lu_r dr + \int_0^t \phi \Gamma^j u_r d \WW^j_r
$$
in $H^3$ for all $\phi \in W^{3, \infty}(\R^d)$. Using Lemma \ref{prodFormula} we get

\begin{align*}
(u_t,\phi u_t) = (g,\phi g) + 2 \int_0^t (Lu_r, \phi u_r)  dr  + 2 \int_0^t (M_r,M_r') d\WW_r
\end{align*}
where 
$$
M_t^j(\phi) = (\Gamma^j u_t, \phi u_t)  \hspace{1cm} (M_t')^{j,i}(\phi) =  (\Gamma^i \Gamma^j  u_t,\phi u_t) + (\Gamma^i u_t, \phi \Gamma^j u_t) 	.
$$
\end{proof}

Let us write this expression more explicitly. Straightforward computations gives
\begin{align*}
2 (Lu, u \phi)  =&  -   ( \partial_j u \partial_i u,   \sigma_{i,k}\sigma_{j,k} \phi ) + 2 ( u^2,   \partial_j \partial_i (\sigma_{i,k}\sigma_{j,k} \phi))    -   (  u^2 ,\partial_j ( b_j  \phi) ) + 2( c    u^2 , \phi)
\end{align*}
and
\begin{align*}
2 (\Gamma^j u_r, \phi u_r) & =  ( \beta_j^n \partial_n u, u \phi) + ( \gamma_j u, u  \phi)   = -  (u^2, \partial_n (\beta_j^n \phi)) + 2 (  u^2, \gamma_j \phi),
\end{align*}
which gives us the equation for $u^2_t$:
\begin{align} \label{squareEquation}
(u_t^2, \phi)   = &  (u_0^2,\phi) +  \int_0^t -   ( \partial_j u_r \partial_i u_r,   \sigma_{i,k}\sigma_{j,k} \phi ) +  ( u^2_r,   \partial_j \partial_i (\sigma_{i,k}\sigma_{j,k} \phi))    -   (  u^2_r ,\partial_j ( b_j  \phi) ) + 2( c    u^2_r , \phi) dr  \notag \\
  & + \int_0^t (u^2_r,  -  \partial_n (\beta_j^n \phi ) +  2\gamma_j  \phi) d\WW^j_r . 
\end{align}

We proceed to find a suitable transformation that allows us to find the energy of the solution $u$.

\begin{lemma}
Assume $\sigma_{i,k}, \gamma_j, \beta_j^n \in C^6_b(\R^d)$ and $b_j, c \in C_b^4(\R^d)$.
Then there exists a solution to the backward equation 
\begin{equation} \label{energyRPDE}
d f_r  =  \left[-   \partial_j \partial_i (\sigma_{i,k}\sigma_{j,k}f_r)    -   \partial_j ( b_j  f_r ) + 2 c f_r    \right] dr + \left[ - \partial_n( \beta_j^n f_r) + 2 \gamma_j f_r \right] d 
\WW_r^j
\end{equation}
with final condition $f_t = 1$ in $W^{3, \infty}(\R^d)$. Moreover, there exists a constant $m>0$ such that 
$$
m^{-1} \leq f_r(x) \leq m
$$
for almost all $r,x$.
\end{lemma}

\begin{proof}
Assume for simplicity that $b_j = c = 0$. 
Existence of a solution to equation \eqref{energyRPDE} is already proven in \cite{DFS}. In fact, it is shown that the solution is in $C^4_b(\R^d)$ and the solution is given by
$$
f_r(x) = E^{(r,x)} \left[ \exp \left\{  \int_r^t 2 \tilde{c}(X_s)ds +  \int_r^t 2 \gamma_j(X_s) -  div \beta_j(X_s)   d\WW_s^j \right\} \right]
$$
where 
$$
dX_s = \tilde{b}(X_s)ds + \sqrt{2}\sigma(X_s) dB_s + \beta(X_s) d\WW_s ,
$$
and we recall the notation \eqref{dualNotation}.

The upper bound, $f_r(x) \leq m$ is already proved in  \cite{DFS}. For the lower bound we argue as follows: for any random variable, $F$, Jensen's inequality gives
$$
(E[ \exp\{ -F \} ])^{-1} \leq E[ \exp\{ F \} ],
$$
thus, the lower bound is proved if we can show that 
$$
\tilde{f}_r(x) := E^{(r,x)} \left[ \exp \left\{ -  \int_r^t 2 \tilde{c}(X_s)ds  -  \int_r^t 2 \gamma_j (X_s) - div \beta (X_s)  d\WW_s^j \right\} \right]
$$ 
is bounded above. This follows by the same way as for the upper bound of $f$. 
\end{proof}

Using the above lemma we will transform \eqref{squareEquation} into an equation where we can easily find the energy estimates. This step should be thought of as the equivalent of the rough Gronwall lemma in \cite{DGHT}, and $f_r(x)$ above as the correct rough exponential to prove this estimate.

\begin{proposition} \label{propEnergyEstimate}
Assume $g$ is smooth. Then the energy estimate of Theorem \ref{thmEnergyEstimate} hold.
\end{proposition}

\begin{proof}
Using Lemma \ref{prodFormula} applied to $u^2$ and $f$ as taking values in $(W^{3, \infty}(\R^d))^*$ and $W^{3, \infty}(\R^d)$ respectively, we get
$$
(u_t^2, 1) = (u_0^2,f_0)  - \int_0^t ( \sigma_{i,k} \sigma_{j,k} \partial_j u_r, \partial_i u_r f_r)  .
$$
Using the upper and lower bounds on $f$ we get
\begin{align*}
\|u_t\|_{L^2}^2  + \lambda m^{-1} \int_0^t \|\nabla u_r\|_{L^2}^2 dr & \leq u_t^2( 1) +  \int_0^t ( \sigma_{i,k} \sigma_{j,k}  \partial_j u_r, \partial_i u_r f_r)  dr = (u_0^2,f_0) \leq m \|u_0\|_{L^2}^2
\end{align*}
where we have used $\lambda |\xi|^2 \leq \sigma_{i,k} \sigma_{j,k} \xi_j \xi_i$. 
\end{proof}

We are now ready to finish the proof of Theorem \ref{thmEnergyEstimate}.

\begin{proof}[Proof of Theorem \ref{thmEnergyEstimate}]
Let $g \in L^2(\R^d)$, and choose $g_n \in C^{\infty}_c(\R^d)$ such that $g_n \rightarrow g$ in $L^2(\R^d)$ and denote by $u^n$ the sequence of solutions corresponding to the initial condition $g_n$. Since the equation is linear, $u^n - u^m$ is a solution to the same equation with the initial condition replaced by $g_n - g_m$. From Proposition \ref{propEnergyEstimate} we see that $u^n$ is a Cauchy sequence in the Banach space $C([0,T]; L^2(\R^d)) \cap L^2([0,T]; H^1)$. Denote by $u$ its limit in these spaces, which from the stability in Theorem \ref{existenceThm} gives that the solution $u$ also satisfies \eqref{energyEstimate}.
\end{proof}

\section{Rough path stability}

In this section we prove that the solution of the backward equation
$$
-du_t = Lu_t + \Gamma^i u_t d\WW^i , \quad u_T = g \in L^2(\R^d) ,
$$
is continuous in $C([0,T]; L^2(\R^d)) \cap L^2([0,T]; H^1)$ under the non-degeneracy condition \eqref{diffusionNonDegeneracy}. We prove this by direct analysis of the Feynman-Kac formula as follows. 

Supposing first that the final condition $g$ is continuous and arguing by rough path stability for RDE's we get continuity of the mapping $ \WW \mapsto E[ g(\Phi_{t,T}(x)]$. This means that the main challenge to prove the desired stability is to show that one can use dominated convergence to conclude
$$
\lim_{ \WW \rightarrow \tilde{\WW}} \int_{\R^d}  ( E[ g(\Phi_{t,T}(x)) ] - E[ g(\tilde{\Phi}_{t,T}(x)) ] )^2 dx   = 0 . 
$$
To do this we show that, uniformly in the rough path metric, we have control on the spread of the mass of the Markov semi-group as follows.

\begin{lemma} \label{Lyaponov}
Define the function $V(x) = e^{-  |x|}$. Then we have 
$$
\int_{\R^d} \sup_{ \|\WW \|_{\alpha} \leq M  ,\,  t \in [0,T]}   E[ V(\Phi_{t,T}(x))]  dx < \infty .
$$
In fact, there exists a constant $C$ such that 
$$
 \sup_{ \|\WW \|_{\alpha} \leq M  , \, t \in [0,T]}   E[ V(\Phi_{t,T}(x))] \leq C e^{ - |x| } .
$$
\end{lemma}

\begin{proof}

From \cite[Lemma 4, Corollary 3]{FR13} we have $|\Phi_{t,T}(x) - x | \leq C( 1 + N_{  [0,T]}( \mathbf{Z}))$ (see Definition \ref{dualLift}).

Using $ - | \Phi_{t,T}(x)| \leq  - |x| + |x - \Phi_{t,T}(x)|$ we get
$$
E[ \exp\{ -|\Phi_{t,T}(x)| \} ] \leq e^{ - |x| } E[ \exp\{ |\Phi_{t,T}(x) - x| \} ] \leq e^{ - |x| } E[ \exp\{ C (1 + N_{  [0,T]}( \mathbf{Z})  ) \} ] .
$$

The result follows immediately since $N_{[0,T]}( \mathbf{Z})$ has Gaussian tails uniformly over bounded sets of $\WW$, see Proposition \ref{summary2}. 

\end{proof}

We start by showing continuity in $L^2(\R^d)$. Notice that we do not use the non degeneracy condition \eqref{diffusionNonDegeneracy} for this result. 

\begin{theorem} \label{thm:L2cont}
Assume $\sigma_{i,k}, \beta_{j} \in C_b^3(\R^d)$, $b_j, \gamma_j, c \in C_b^1(\R^d)$. Then the solution map  
\begin{align*}
L^2(\R^d) \times \mathscr{C}_{g}^{\alpha} & \rightarrow C([0,T] ; L^2(\R^d)) \\
(g,\WW) & \mapsto u 
\end{align*}
is continuous.
\end{theorem}

\begin{proof}

Step 1: we first fix $g \in C_c(\R^d)$ and show the rough path stability $\WW \mapsto u$. 
By the rough path stability $ \WW \mapsto \Phi_{t,T}(x)$ we have that 
$$
\lim_{ \WW \rightarrow \tilde{\WW}} \sup_{t \in [0,T]}  ( E[ g(\Phi_{t,T}(x)) ] - E[ g(\tilde{\Phi}_{t,T}(x)) ] )^2   = 0 .
$$
Thus, to show that 
$$
\lim_{ \WW \rightarrow \tilde{\WW}} \int_{\R^d} \sup_{t \in [0,T]}  ( E[ g(\Phi_{t,T}(x)) ] - E[ g(\tilde{\Phi}_{t,T}(x)) ] )^2 dx   = 0
$$
it is enough to show that we may use dominated convergence. Since $g$ has compact support, it is clear that $e^{ |x|} g(x)$ is uniformly bounded in $x$. We get
\begin{align*}
E[ g(\Phi_{t,T}(x)) ] & = E[ g(\Phi_{t,T}(x)) \exp\{  |\Phi_{t,T}(x)| \} \exp\{ -  |\Phi_{t,T}(x)| \} ] \\
& \leq \| e^{ | \cdot |} g( \cdot ) \|_{\infty} E[\exp\{ -  |\Phi_{t,T}(x)| \} ] \\
& \leq \| e^{ | \cdot |} g( \cdot ) \|_{\infty} C \exp\{ -  |x|\} .
\end{align*}

Step 2: Let $g,\tilde{g} \in L^2(\R^d)$ and $\epsilon > 0$. Choose $g_{\epsilon}, \tilde{g}_{\epsilon} \in  C_c(\R^d)$ such that 
$$
\|g - g_{\epsilon} \|_{0} \leq \epsilon, \quad \|\tilde{g} - \tilde{g}_{\epsilon} \|_0 \leq \epsilon  \quad \textrm{and} \quad  \|g_{\epsilon} - \tilde{g}_{\epsilon} \|_0 \leq \|g - \tilde{g} \|_{0}.
$$

Since the mapping $g \mapsto u$ is linear, we get from \eqref{InitialLinear} that 
$$
\sup_{t \in [0,T]} \int_{\R^d} \left( E[ g( \Phi_{t,T}(x)) ] - E[ g_{\epsilon}( \Phi_{t,T}(x)) ] \right)^2 dx \leq C \epsilon^2 , 
$$
$$
\sup_{t \in [0,T]} \int_{\R^d} \left( E[ g_{\epsilon}( \tilde{\Phi}_{t,T}(x)) ] - E[ \tilde{g}_{\epsilon}( \tilde{\Phi}_{t,T}(x)) ] \right)^2 dx \leq C \|g - \tilde{g} \|_{0}^2 . 
$$
This gives
\begin{align*}
\lim_{ \WW \rightarrow \tilde{\WW}} \sup_{t \in [0,T]}  \int_{\R^d} &  ( E[ g(\Phi_{t,T}(x)) ] - E[ \tilde{g}(\tilde{\Phi}_{t,T}(x)) ] )^2 dx  \\
 & \lesssim  \lim_{ \WW \rightarrow \tilde{\WW}} \sup_{t \in [0,T]}  \int_{\R^d}   ( E[ g(\Phi_{t,T}(x)) ] - E[ g_{\epsilon}(\Phi_{t,T}(x)) ] )^2 dx  \\
&  \lim_{ \WW \rightarrow \tilde{\WW}} \sup_{t \in [0,T]}  \int_{\R^d}   ( E[ g_{\epsilon}(\Phi_{t,T}(x)) ] - E[ g_{\epsilon}(\tilde{\Phi}_{t,T}(x)) ] )^2 dx  \\ 
& + \lim_{ \WW \rightarrow \tilde{\WW}} \int_{\R^d}  \sup_{t \in [0,T]}  ( E[ g_{\epsilon}(\tilde{\Phi}_{t,T}(x)) ] - E[ \tilde{g}_{\epsilon}(\tilde{\Phi}_{t,T}(x)) ] )^2 dx \\
& +  \lim_{ \WW \rightarrow \tilde{\WW}} \sup_{t \in [0,T]}  \int_{\R^d}   ( E[ \tilde{g}_{\epsilon}(\tilde{\Phi}_{t,T}(x)) ] - E[ \tilde{g}(\tilde{\Phi}_{t,T}(x)) ] )^2 dx \\
 & \leq  2C \epsilon^2  + C \|g - \tilde{g} \|_{0}^2 .
\end{align*}

Since $\epsilon $ was arbitrary, the result follows.
\end{proof}

\begin{theorem} \label{thm:H1cont}
Suppose $\sigma_{i,k}, \gamma_j , \beta_j^n \in C_b^6(\R^d)$, $b_j,c \in C_b^4(\R^d)$ as well as the non degeneracy condition \eqref{diffusionNonDegeneracy}. Then the solution map
\begin{align*}
L^2(\R^d) \times \mathscr{C}_g^{\alpha} & \rightarrow  L^2([0,T];  H^1) \\
(g,\WW)  & \mapsto u 
\end{align*}
is continuous.
\end{theorem}

\begin{proof}

The proof is similar to the proof of Theorem \ref{thm:L2cont}, except we will use the energy estimates \eqref{energyEstimate} instead of \eqref{InitialLinear}.

Step 1: we first fix $g \in C_c^1(\R^d)$ and show the rough path stability $\WW \mapsto u$.
We have
$$
\nabla u_t(x) = E[ \nabla g( \Phi_{t,T}(x)) \nabla \Phi_{t,T}(x) ]
$$

By the rough path stability $ \WW \mapsto \Phi_{t,T}(x)$ and $\WW \mapsto \nabla \Phi_{t,T}(x)$ we have that 
$$
\lim_{ \WW \rightarrow \tilde{\WW}} \sup_{t}  ( E[ \nabla g( \Phi_{t,T}(x)) \nabla \Phi_{t,T}(x) ] - E[ \nabla g( \tilde{\Phi}_{t,T}(x)) \nabla \tilde{\Phi}_{t,T}(x) ] )^2   = 0 .
$$
Thus, to show that 
$$
\lim_{ \WW \rightarrow \tilde{\WW}} \int_0^T \int_{\R^d}  ( E[ \nabla g( \Phi_{t,T}(x)) \nabla \Phi_{t,T}(x) ] - E[ \nabla g( \tilde{\Phi}_{t,T}(x)) \nabla \tilde{\Phi}_{t,T}(x) ] )^2dx dt  = 0
$$
it is enough to show that we may use dominated convergence. Since $g$ has compact support, it is clear that $e^{ |x|} |\nabla g(x)|^2$ is uniformly bounded in $x$. We get
\begin{align*}
E[ \nabla g(\Phi_{t,T}(x)) \nabla \Phi_{t,T}(x) ]^2 & \leq E[ |\nabla g(\Phi_{t,T}(x)) |^2] E[ |\nabla \Phi_{t,T}(x)|^2 ] \\
 & \leq E[ |\nabla g(\Phi_{t,T}(x)) |^2 \exp\{  |\Phi_{t,T}(x)| \} \exp\{ -  |\Phi_{t,T}(x)| \} ] E[ |\nabla \Phi_{t,T}(x)|^2 ]\\
& \leq \| e^{ | \cdot |} |\nabla g( \cdot )|^2 \|_{\infty} E[\exp\{ -  |\Phi_{t,T}(x)| \} ] E[ |\nabla \Phi_{t,T}(x)|^2 ]\\
& \leq \| e^{ | \cdot |} |\nabla g( \cdot )|^2 \|_{\infty} C \exp\{ -  |x| \} .
\end{align*}

Step 2: Let $g,\tilde{g} \in L^2(\R^d)$ and $\epsilon > 0$. Choose $g_{\epsilon},\tilde{g}_{\epsilon} \in C_c^1(\R^d)$ as in the proof of Theorem \ref{thm:L2cont}. Define the functions 
$$
u_t^{\epsilon}(x) = E[g_{\epsilon}(\Phi_{t,T}(x))],  \quad \tilde{u}_t^{\epsilon}(x) = E[\tilde{g}_{\epsilon}(\Phi_{t,T}(x))]   \quad \textrm{and} \quad  v^{\epsilon}_t(x) = E[g_{\epsilon}(\tilde{\Phi}_{t,T}(x))] .
$$
Since the equation is linear, we get from Theorem \ref{thmEnergyEstimate} that 
$$
\int_0^T \| \nabla  u_t -  \nabla u_t^{\epsilon} \|_{0}^2 dt \leq C \epsilon^2 , \quad \int_0^T \| \nabla  \tilde{u}_t -  \nabla \tilde{u}_t^{\epsilon} \|_{0}^2 dt \leq C \epsilon^2
$$
and
$$
\int_0^T \| \nabla  v^{\epsilon}_t -  \nabla \tilde{u}_t^{\epsilon} \|_{0}^2 dt \leq C \|g - \tilde{g}\|_{0}^2.
$$
This gives
\begin{align*}
\lim_{ \WW \rightarrow \tilde{\WW}} \int_0^T &  \| \nabla  u_t -  \nabla \tilde{u}_t \|_{0}^2 dt  \lesssim \lim_{ \WW \rightarrow \tilde{\WW}}   \int_0^T \| \nabla  u_t -  \nabla u_t^{\epsilon} \|_{0}^2 dt  + \lim_{ \WW \rightarrow \tilde{\WW}}  \int_0^T \| \nabla u_t^{\epsilon} -  \nabla v_t^{\epsilon} \|_{0}^2 dt \\
 &  + \lim_{ \WW \rightarrow \tilde{\WW}} \int_0^T \| \nabla v_t^{\epsilon} -  \nabla \tilde{u}_t^{\epsilon} \|_{0}^2 dt + \lim_{ \WW \rightarrow \tilde{\WW}}  \int_0^T \| \nabla \tilde{u}_t^{\epsilon} -  \nabla \tilde{u}_t \|_{0}^2 dt \\
  & \leq 2 C \epsilon^2 + C \|g - \tilde{g} \|_{0}^2.
\end{align*}

Since $\epsilon $ was arbitrary, the result follows.
\end{proof}

\begin{remark}
We note from the proofs of Theorem \ref{thm:L2cont} and Theorem \ref{thm:H1cont} that we have a Lipschitz-type continuity in the initial condition $g$. 
\end{remark}

\section{Semi-linear perturbation}

In this section we study the semi-linear equation
\begin{equation} \label{eq:non-linear}
du_t = \left( Lu_t  + F( u_t) \right] dt  + \Gamma^i u_t d\WW^i_t ,  \quad u_0  = g \in L^2(\R^d)
\end{equation}
for some non-linearity $F: H^1 \rightarrow L^2(\R^d)$. 
A solution to \eqref{eq:non-linear} is defined in a similar way as the analytically weak solution in Definition \ref{def:strongForwardSolutionRough}, except we need more regularity on the solution to make sense of $F(u)$.  

\begin{definition}
We say that $u \in C([0,T]; L^2(\R^d)) \cap L^2([0,T]; H^1)$ is an analytically weak solution to \eqref{eq:non-linear} if $(\Gamma^i u, \Gamma^i \Gamma^j u)$ is controlled by $W$ in $H^{-3}$ and for all $\phi \in H^3$ we have
$$
(u_t, \phi) = (u_0, \phi) + \int_0^t (u_s, L^* \phi)  + (F(u_s), \phi) ds + \int_0^t (u_s, \Gamma^{i,*} \phi) d\WW_s^i .
$$
\end{definition}

The rest of this section is devoted to proving the following well-posedness result.

\begin{theorem} \label{thm:Non-linearWellPosed}
Suppose $\sigma_{i,k}, \gamma_j , \beta_j^n \in C_b^6(\R^d)$, $b_j,c \in C_b^4(\R^d)$ as well as the non degeneracy condition
$$
\lambda |\xi|^2 \leq \sigma_{i,k} \sigma_{j,k} \xi_j \xi_i .
$$
Assume the non linearity is Lipschitz, from $H^1$ to $L^2(\R^d)$, i.e. 
\begin{equation} \label{FConditions}
\| F(u)  - F(v) \|_0 \lesssim \|u - v\|_1, \qquad \forall u,v \in H^1 .
\end{equation}
Then there exists a unique solution to \eqref{eq:non-linear}.
\end{theorem}
Uniqueness and existence will be proven separately, in Proposition \ref{thm:VariationalUniqueness} and Proposition \ref{thm:VariationalExistence}, respectively.

The proof will be based on Duhamel's principle, i.e. in the mild formulation of the equation. The twist is that we will consider the two parameter semi-group generated by the diffusive {\em and} the rough terms. More precisely, if we denote by $P_{s \cdot}^{\WW}$ the operator 
\begin{align*}
L^2(\R^d) & \rightarrow C([s,T]; L^2(\R^d)) \cap L^2([s,T]; H^1) \\
 g & \mapsto v
\end{align*}
where $v$ denotes the solution to \eqref{eq:non-linear} with $F =0$ starting at time $s$ in $g$. The Duhamel's principle in the classical setting then tells us that an equivalent formulation of \eqref{eq:non-linear} is given by
\begin{equation} \label{eq:MildFormulation}
u_t = P_{0 t}^{\WW} g + \int_0^t P_{st}^{\WW} F(u_s) ds
\end{equation}
in $L^2(\R^d)$. As soon as $P^{\WW}$ is defined, the notion of solution to the above equation, as well as its well-posedness, will be standard. 

\begin{definition}
A mapping $u \in C([0,T]; L^2(\R^d))$ is a solution to \eqref{eq:MildFormulation} provided $u \in  L^2([0,T]; H^1)$ and the equality \eqref{eq:MildFormulation} holds in $C([0,T]; L^2(\R^d))$.
\end{definition}

We start by showing well-posedness of \eqref{eq:MildFormulation}.

\begin{proposition} \label{thm:MildWellPosedness}
Retain the assumptions of Theorem \ref{thm:Non-linearWellPosed}. Then there exists a unique solution of the mild formulation, equation \eqref{eq:MildFormulation}.
\end{proposition}

\begin{proof}
We set up a contraction mapping on the space
$$
\mathcal{X} = C([0,T]; L^2(\R^d)) \cap L^2([0,T]; H^1),
$$
via $\Xi :\mathcal{X}  \rightarrow \mathcal{X}$ as
$$
\Xi(u)_t = P_{0t}^{\WW} u_0 + \int_0^t P_{st}^{\WW} F(  u_s) ds,
$$
where the initial condition $u_0$ is fixed. 
We first show that $\Xi$ is well defined. From Theorem \ref{thmEnergyEstimate} we get $P_{0 \cdot}^{\WW} u_0 \in \mathcal{X}$.
Moreover, it is clear that we have $\int_0^{\cdot} P_{s \cdot }^{\WW} F( u_s) ds \in C([0,T]; L^2(\R^d))$, and we have the bound
\begin{align*}
\sup_{t \in [0,T] } \left\| \int_0^{t} P_{s t }^{\WW} F( u_s) ds \right\|_0 & \leq \sup_{t \in [0,T] }  \int_0^{t} \left\| P_{s t }^{\WW} F( u_s) \right\|_0 ds \lesssim  \sup_{t \in [0,T] }  \int_0^{t} \left\| F(u_s) \right\|_0 ds \\
 &  \leq T^{1/2} \left( \int_0^{T} \left\| F(u_s) \right\|^2_0 ds \right)^{1/2} \leq  T^{1/2} \left( \int_0^{T} (1 +  \left\| u_s \right\|_1^2 ) ds\right)^{1/2}
\end{align*}

To see that $\Xi$ also takes values in $L^2([0,T]; H^1)$, consider
\begin{align*}
\int_0^T \left\| \nabla \int_0^t P_{st}^{\WW} F(  u_s) ds \right\|_0^2 dt & \leq \int_0^T t \int_0^t \| \nabla  P_{st}^{\WW} F(u_s)\|_0^2  ds  dt  \leq T \int_0^T  \int_t^T \| \nabla  P_{st}^{\WW} F(u_s)\|_0^2  dt  ds \\
 & \lesssim T \int_0^T   \|  F( u_s)\|_0^2  dt  \lesssim T( 1 + \int_0^T \|u_s \|_1^2 ds) 
\end{align*}

Finally, to see that $\Xi$ is a contraction, we write for $u, v \in \mathcal{X}$ using similar computations as above
\begin{align*}
\sup_{t \in [0,T] } \left\| \int_0^{t} P_{s t }^{\WW} ( F(  u_s) - F(v_s)) ds   \right\|_0 & 
   \lesssim   T^{1/2} \left( \int_0^{T} \left\| u_s  - v_s \right\|_1^2 ) ds\right)^{1/2}
\end{align*}
and 
\begin{align*}
\int_0^T \left\| \nabla \int_0^t P_{st}^{\WW} ( F( u_s) - F( v_s))  ds \right\|_0^2 dt  & \lesssim T \int_0^T   \|  F( u_s) - F( v_s) \|_0^2  dt  \lesssim T  \int_0^T \|u_s  - v_s\|_1^2 ds 
\end{align*}
which shows that 
$$
\| \Xi(u) - \Xi(v) \|_\mathcal{X} \lesssim T^{1/2} \| u - v \|_\mathcal{X} .
$$
It follows that $\Xi$ is a contraction mapping with a unique fix point provided $T$ is small enough. It is now standard to extend the solution on a general time interval $[0,T]$ by iteration. 
\end{proof}

We go on to show that a variational solution, \eqref{eq:non-linear}, is a solution to \eqref{eq:MildFormulation}.

\begin{proposition} \label{thm:VariationalUniqueness}
Assume $\sigma_{i,k}, \gamma_j , \beta_j^n \in C_b^6(\R^d)$, $b_j,c \in C_b^4(\R^d)$. Then a solution to \eqref{eq:non-linear} is a solution to \eqref{eq:MildFormulation}. In particular, using Proposition \ref{thm:MildWellPosedness}, solutions to \eqref{eq:non-linear} are unique. 
\end{proposition}

\begin{proof}
With Lemma \ref{prodFormula} this is straightforward. Assume $u$ is a solution to \eqref{eq:non-linear}. Let $\phi \in C^{\infty}_c(\R^d)$ and denote by $v$ the backward solution to
$$
- d v_s = L^*v_s ds + \Gamma^{j,*} v_s d\WW_s^j, \quad v_t = \phi.
$$
From Lemma \ref{prodFormula}, similarly as in the proof of Theorem \ref{Uniqueness}, we get
$$
(u_t , v_t) = (u_0, v_0)  + \int_0^t  (F(\nabla u_s), v_s) ds
$$
and by noticing that we may write $v_s = P_{st}^{\WW,*} \phi$ where $P_{st}^{\WW,*}$ is the adjoint of $P_{st}^{\WW}$, we get
\begin{align*}
(u_t , \phi) & = (u_0, P_{0 t}^{\WW,*} \phi)  + \int_0^t  (F( u_s), P_{st}^{\WW,*} \phi) ds \\
 & = ( P_{0,t}^{\WW} u_0,  \phi)  + \int_0^t  ( P_{st}^{\WW} F(u_s),  \phi) ds .
\end{align*}
Since $\phi$ was arbitrary in $C^{\infty}_c(\R^d)$ the result follows. 
\end{proof}

Note that even though we do not prove equivalence of the two definitions, the one implication is enough to show well-posedness of \eqref{eq:non-linear}. Indeed, since we have shown that \eqref{eq:non-linear} implies \eqref{eq:MildFormulation} this immediately gives uniqueness of \eqref{eq:non-linear}. 

To show existence of a solution to \eqref{eq:non-linear} we argue by rough path continuity. For a smooth path the formulations are classically equivalent. To take the limit in the approximation we shall use a compactness criterion also used in \cite{HNS}. The computations below will show precisely the missing ingredient to show that the definitions are equivalent and also hint at how this could be solved, see Remark \ref{rmk:p-variation}.

\begin{proposition} \label{thm:VariationalExistence}
Under the assumptions of Theorem \ref{thm:Non-linearWellPosed} there exists a solution to \eqref{eq:non-linear}.
\end{proposition}

\begin{proof}
Assume first that $W$ is smooth. With the notation of the previous proposition, notice first that similar (but easier) computations as in the proof of Theorem \ref{existenceThmRegular} shows that for $\phi \in H^1$ we have
\begin{equation} \label{ForwardTimeReverseFlow}
\|\delta ( P_{r \cdot}^{\WW,*} \phi)_{st} -  P_{r s}^{\WW,*} \Gamma^{i,*} \phi W_{st}^i \|_0 \lesssim |t-s|^{2 \alpha} \|\phi \|_1 .
\end{equation}

Since $W$ is smooth, $u$ satisfies \eqref{eq:MildFormulation}, which gives that for any $\phi \in H^1$
\begin{align}
(\delta  u_{st}, \phi) &  = ( \delta (P_{0 \cdot}^{\WW} g)_{st}, \phi) + \int_0^s (  \delta ( P_{r \cdot}^{\WW} F(u_r))_{st}  , \phi)dr  + \int_s^t (  P_{r t }^{\WW} F(u_r),  \phi) dr   \notag \\
&  = (  g, \delta (P_{0 \cdot}^{\WW,*} \phi)_{st} ) + \int_0^s (   F(u_r)  , \delta (P_{r \cdot}^{\WW,*}\phi)_{st})dr  + \int_s^t (  P_{r t }^{\WW} F(u_r),  \phi) dr . \label{eq:mildExpansion}  
\end{align}

From \eqref{ForwardTimeReverseFlow} we get
\begin{align*}
\left|(  g, \delta (P_{0 \cdot}^{\WW,*}\phi)_{st} ) - (  g,P_{0 s}^{\WW,*} \Gamma^{i,*} \phi)  W_{st}^i \right| \lesssim |t-s|^{2 \alpha} \|\phi\|_1 \|g\|_0 ,
\end{align*}
\begin{align*}
\int_0^s \left| ( F(u_r) , \delta (P_{r \cdot}^{\WW,*}\phi)_{st} ) - ( F(u_r)  ,P_{r s}^{\WW,*} \Gamma^{i,*} \phi)  W_{st}^i \right| ds  \lesssim |t-s|^{2 \alpha} \|\phi\|_1 \int_0^s \| F(u_r) \|_0 dr.
\end{align*}
and
\begin{equation} \label{p-variationNeeded}
\left|\int_s^t (  P_{r t }^{\WW} F(u_r),  \phi) dr \right| \leq \left(\int_s^t \| F(u_r) \|_0^2 dr \right)^{1/2} (t-s)^{1/2}  \lesssim \left(1 + \int_0^T \|u_r \|_1^2 dr \right)^{1/2} (t-s)^{1/2} 
\end{equation}
uniformly in the rough path metric. This shows that
\begin{equation} \label{uniformHolder}
\sup_{\| \WW \|_{\alpha} \leq M} \| \delta u_{st} \|_{-1} \lesssim \left(1 + \int_0^T \| u_r \|_1^2 dr \right)^{1/2} |t-s|^{1/2} .
\end{equation}

This shows that uniformly over bounded sets of $\| \WW \|_{\alpha}$, the solution $u$ remains in a bounded set of 
$$
C([0,T]; L^2(\R^d)) \cap L^2([0,T]; H^1) \cap C^{1/2}([0,T]; H^{-1}) 
$$
which is compactly embedded into
$$
L^2([0,T]; L^2(K)) \cap C([0,T]; L^{2}(K)_w) 
$$
for any compact $K \subset \R^d$, see e.g. \cite[Lemma A.1]{HNS}. In the above we have denoted by $L^{2}(K)_w$ the usual $L^2$ space equipped with its weak topology. 

Now, for any geometric rough path $\WW$, we take a sequence $W^n$ of smooth paths converging to $\WW$ in the rough path metric. Then the sequence of corresponding solution $u^n$ is relatively compact in $L^2([0,T]; L^2(K)) \cap C([0,T]; L^{2}(K)_w)$. We may thus take a sub-sequence converging to $u$ and taking the limit in the equation \eqref{eq:non-linear} we obtain a solution. See \cite[Theorem 4.2]{HNS} for precise details.

\end{proof}

\begin{remark} \label{rmk:p-variation}
We note that in order to show that a solution to \eqref{eq:MildFormulation} is controlled by $W$ in $H^{-1}$, one would proceed as in the above proof up to \eqref{p-variationNeeded}. However, the estimate in \eqref{p-variationNeeded} is not enough to show that the term 
$
\int_s^t  P_{r t }^{\WW} F(u_r) dr
$
showing up in the expansion \eqref{eq:mildExpansion}
is a remainder in $H^{-1}$ in the sense of controlled paths, see Definition \ref{def:ControlledPath}. One could circumvent this by introducing $p$-variation spaces which would give a direct approach to show that the mild and variational formulations actually are equivalent. 
\end{remark}

\section{Appendix}

The main objective of the appendix is to develop the theory of controlled rough paths for compositions by $L^p(\R^d)$-functions needed for the paper. This is not possible for a general controlled path $(Y,Y') \in \mathscr{D}_Z^{2 \alpha}([0,T]; C(\R^d;\R^d))$ without some sort of non-degeneracy condition on the spatial variable. Indeed, assume e.g. $(Y,Y') = (Z,I)$ is constant in space. Then for any $g \neq 0$, 
$$
\int g(Y)^p dx = \infty,
$$
so that there is no hope in giving meaning to $(g(Y), g(Y)') = (g(Y), \nabla g(Y) Y') \in \D^{2 \alpha}_Z([0,T]; L^p(\R^d))$.

We introduce the appropriate notion of non-degeneracy which prevent this situation.

\begin{definition} \label{nonDegenerateControlledPath}
We shall say that a controlled path $(Y,Y') \in \mathscr{D}_{Z}^{2 \alpha}([0,T]; C(\R^d;\R^d)$ is $C^1$-non-degenerate provided $Y$ is actually $C^1( \R^d;\R^d)$-valued and 
\begin{equation} \label{flowDeterminantCondition}
\ND(Y) := \inf_{s,t \in [0,T], x \in \R^d, \theta \in [0,1]}  det(\nabla ( Y_s(x) + \theta \delta Y_{st}(x) ) ) > 0. 
\end{equation}
\end{definition}

\begin{lemma} \label{L2Composition}
Let $g \in W^{2,p}(\R^d)$ and assume $(Y,Y') \in \D_Z^{2 \alpha}([0,T]; C_b(\R^d;\R^d))$ satisfies $\ND(Y) > 0$. Then we have 
$$
\left( g(Y) , \ \nabla g(Y)Y'  \right) \in \D^{2 \alpha}_Z([0,T]; L^p(\R^d))
$$
with the bound 
\begin{align*}
\|( g(Y) , \nabla g(Y) Y' )  \|_{\alpha, Z; L^p(\R^d)}  & \leq   C \ND(Y)^{-1/p} ( 1 + \|Z\|_{\alpha})^2 (  \|Y_0'\|_{\infty} + \|(Y,Y')\|_{\alpha, Z;C_b(\R^d;\R^d)})^2 |g|_{2,p} 
\end{align*}
\end{lemma}

\begin{proof}

We begin by showing the estimates for smooth $g$, and the general case will follow by taking approximations.

Notice first that $y \mapsto  g(Y_t(y))  $ is a well defined element of $L^p(\R^d)$;
\begin{align} \label{CompositionWellDef}
 \int g(Y_t(y))^p dy  &   = \int g(x)^p det( \nabla Y_t(x))^{-1} dx  \leq \ND(Y)^{-1}  \int g(x)^p  dx  
\end{align}
where we have used $ det( \nabla Y_t(x)) \geq \ND(Y)$.

For the Gubinelli derivative, we consider
\begin{align*}
 \delta \left( \nabla g(Y_{\cdot}(y))   Y_{\cdot}'(y)  \right)_{st}  & =   \delta (\nabla g(Y_{\cdot}(y)))_{st}   Y_{s}'(y)  +  \nabla g(Y_t(y))   \delta Y_{st}'(y)  .
\end{align*}
For the first term we write
\begin{align*}
\delta (\nabla g(Y_{\cdot}(y)))_{st} Y_s'(y)  =  \int_0^1 \nabla^2 g(Y_s(y) + \theta \delta Y_{st}(y) ) d \theta \delta Y_{st}(y) Y_s'(y)
\end{align*}
which gives
\begin{align*}
\int |\delta (\nabla g(Y_{\cdot}(y)))_{st}|^p |Y_s'(y)|^p dy  & \leq  |Y'(y)|^p_{\infty} \int_0^1 \int |\nabla^2 g(Y_s(y) + \theta \delta Y_{st}(y) )  \delta Y_{st}(y)|^p dy d \theta \\
 & \leq |t-s|^{p \alpha}  ( \|Y_0'\|_{C_b(\R^d;\R^d)} + T^{ \alpha} \|Y'\|_{\alpha})^{2p} \ND(Y)^{-1} \|\nabla^2 g\|_{0,p}^p ,
\end{align*}
where we have used
\begin{equation} \label{supToHolder}
|Y_t'(y) | \leq |Y_t'(y) - Y_0'(y) | + |Y_0'(y)| \leq T^{\alpha} \|Y' \|_{\alpha}+ |Y'_0(y)| .
\end{equation}
For the second term use the bound
$$
\sup_{y \in \R^d }\| Y'(y) \|_{\alpha}^p |t-s|^{p \alpha}  \int |\nabla g(Y_{\cdot}(y))|^p dy     \leq \| (Y,Y') \|^p_{\alpha, Z;C_b(\R^d;\R^d)} \ND(Y)^{-1}\|\nabla g\|_{0,p}^p  |t-s|^{p \alpha} .
$$
For the remainder we use Taylor's formula
$$
g(x) - g(y) =  \nabla g(y) (x- y)  + \int_0^1 \nabla^2 g(y + \theta (x-y))(1- \theta ) d \theta (x-y)^{\otimes 2}
$$
to get 
\begin{align*}
 \delta   g(Y_{\cdot}(y))_{st} & =  \nabla g(Y_s(y)) \delta Y_{st}(y)   +   \int_0^1 \nabla^2 g(Y_s(y) + \theta \delta Y_{st}(y) )(1- \theta )   d \theta  ( \delta Y_{st}(y))^{\otimes 2} \\
  & =   \nabla g(Y_s(y)) Y_s'(y)     \delta Z_{st} +\nabla g(Y_s(y)) \left( \delta Y_{st}(y) - Y_s'(y) \delta Z_{st} \right)   \\
  & + \int_0^1 \nabla^2 g(Y_s(y) + \theta \delta Y_{st}(y) )(1- \theta )   d \theta  ( \delta Y_{st}(y))^{\otimes 2} . 
\end{align*}

For the second term we use the bound
\begin{align*}
\int |\nabla g(Y_s(y)) \left( \delta Y_{st}(y) - Y_s'(y) \delta Z_{st}  \right) |^p    dy &  \leq |t-s|^{2p \alpha} \sup_{y \in \R^d} \| \delta Y(y) - Y'(y) \delta Z \|_{2 \alpha}^{p}  \int |\nabla g(Y_{\cdot}(y))|^p dy  \\
&  \leq |t-s|^{2p \alpha}  \| (Y, Y') \|^p_{\alpha,Z;C_b(\R^d;\R^d)} \ND(Y)^{-1} \int |\nabla g(x)|^p dx  . 
\end{align*}

For the last term we use the bound
\begin{align*}
 \int \int_0^1 |\nabla^2 &  g(Y_s(y) + \theta \delta Y_{st}(y) )(1- \theta )   ( \delta Y_{st}(y))^{\otimes 2} |^p d \theta dy  \\
  & \leq |t-s|^{2p \alpha} \sup_{y \in \R^d} \| Y(y) \|^{2p}_{\alpha} \int_0^1 \int  |\nabla^2 g(Y_s(y) + \theta \delta Y_{st}(y) )  |^p  dy  (1 - \theta)^2 d \theta \\
  & \leq |t-s|^{2p \alpha} \sup_{y \in \R^d} \| Y(y) \|^{2p}_{ \alpha}   \int_0^1 \int  |\nabla^2 g(x )  |^p det( \nabla (Y_s(x) + \theta \delta Y_{st}(x) ))^{-1} dy  d \theta \\
  & \leq |t-s|^{2p \alpha} \sup_{y \in \R^d} \| Y(y) \|^{2p}_{ \alpha}  \ND(Y)^{-1} \|\nabla^2 g  \|_{0,p}^p .
 \end{align*}

As in \eqref{supToHolder} we have
\begin{align*}
 \| Y(y) \|_{ \alpha} &  \leq \sup_{t \in [0,T]} |Y_t'(y)| \|Z\|_{\alpha} + \|  \delta Y - Y' \delta Z  \|_{2 \alpha} \leq (|Y_0'(y)| + T^{\alpha} \|Y'(y)\|_{\alpha} )\|Z\|_{\alpha} + \| \delta Y - Y' \delta Z\|_{2 \alpha} \\
&  \leq \left( |Y_0'(y)|  + \|(Y,Y')\|_{\alpha, Z;C_b(\R^d;\R^d)} \right)( T^{\alpha}  \|Z\|_{\alpha} + 1) . 
\end{align*}

For general $g \in W^{2,p}(\R^d)$, take a smooth approximation $g_n$ such that $g_n \rightarrow g$ in $W^{2,p}(\R^d)$. The condition \eqref{flowDeterminantCondition} allows us to take the $L^p(\R^d)$-limit in all the expressions above, e.g.
$$
\int |(g_n -g)(Y_t(y))|^p dy \leq \ND(Y)^{-1} \int |g_n(x) -  g(x)|^p dx \rightarrow 0
$$
and
\begin{align*}
\int \left| \int_0^1 \nabla^2 (g_n - g) (Y_s(y) + \theta \delta Y_{st}(y) ) d \theta \right|^p dy & \leq \int_0^1 \int \left|  \nabla^2 (g_n - g) (Y_s(y) + \theta \delta Y_{st}(y) )  \right|^p dy d \theta \\
& \leq \ND(Y)^{-1} \int |\nabla^2 g_n(x) -  \nabla^2 g(x)|^p dx \rightarrow 0 .
\end{align*}
\end{proof}

For technical reasons, in Theorem \ref{Uniqueness} and Lemma \ref{SquaredEquation} we shall need also to construct the rough path integral as an $W^{3,p}(\R^d)$-valued object. This necessitates more regularity on the coefficients. The proof follows from a tedious generalization of Lemma \ref{L2Composition}.

\begin{lemma} \label{L2CompositionH3}
Let $g \in W^{5,p}(\R^d)$ and assume $(Y,Y') \in \D_Z^{2 \alpha}([0,T]; C^3_b(\R^d;\R^d))$ satisfies $\ND(Y) > 0$. Then we have 
$$
\left( g(Y) , \ \nabla g(Y)Y'  \right) \in \D^{2 \alpha}_Z([0,T]; W^{3,p}(\R^d))
$$
with the bound 
\begin{align*}
\|( g(Y) , \nabla g(Y) Y' )  \|_{\alpha, Z; W^{3,p}(\R^d)}  & \leq   C \ND(Y)^{-1/p} ( 1 + \|Z\|_{\alpha})^2  \\
 & \qquad \times ( 1 +  \|Y_0'\|_{C_b^3(\R^d)} + \|(Y,Y')\|_{\alpha, Z; C_b^3(\R^d)})^2  \|g\|_{5,p} .
\end{align*}
\end{lemma}

\begin{proof}
From Lemma \ref{L2Composition} we know that for any $h \in L^q(\R^d)$ we have
\begin{align} \label{innerProduct}
\left| \left( \delta g(Y)_{st} - \nabla g(Y_s) Y_s' \delta Z_{st} ,  h \right) \right| & \lesssim |t-s|^{2 \alpha} \|g\|_{2,p} \|h\|_{0,q} \ND(Y)^{-1/p} ( 1 + \|Z\|_{\alpha})^2 \notag \\
 & \qquad \times( 1 +  \|Y_0'\|_{C_b(\R^d)} + \|(Y,Y')\|_{\alpha, Z; C_b(\R^d)})^2 .
\end{align}
Replacing $h$ by $\nabla h$ in the above inner product we get
\begin{align*}
\big( \nabla( \delta g(Y& )_{st}   - \nabla g(Y_s) Y_s' \delta Z_{st}) , h \big)  =  \left( ( \delta \nabla g(Y) \nabla Y)_{st} - \nabla^2 g(Y_s) \nabla Y_s \otimes Y_s' \delta Z_{st}- \nabla g(Y_s) \nabla  Y_s' \delta Z_{st} ,  h \right) \\
 & =  \left(  (\delta \nabla g(Y)_{st} - \nabla^2 g(Y_s) Y_s' \delta Z_{st}) \nabla Y_s  + \nabla g(Y_{s})(  \delta \nabla Y_{st}  -  \nabla  Y_s' \delta Z_{st})  + \delta \nabla g(Y)_{st} \delta \nabla Y_{st} , h \right) \\
& \leq \| \nabla Y\|_{\infty} \|\delta \nabla g(Y)_{st} - \nabla^2 g(Y_s) Y_s' \delta Z_{st} \|_{0,p} \| h\|_{0,q} \\
& + \| \nabla g(Y_s) \|_{0,p} \|h\|_{0,q} \|\delta Y_{st} - Y_s' \delta Z_{st} \|_{C^1_b(\R^d)} + \|\delta Y_{st}\|_{C^1_b(\R^d)} \| \delta \nabla g(Y)_{st} \|_{0,p} \|h\|_{0,q}
\end{align*}
The first term can be bounded using Lemma \ref{L2Composition}. The second term is bounded by the assumption that $(Y,Y')$ is controlled in $C^3_b(\R^d)$. The last term is bounded by Lemma \ref{L2Composition};
\begin{align*}
\| \delta \nabla g(Y)_{st} \|_{0,p} & \leq \| \delta \nabla g(Y)_{st}  - \nabla^2g(Y_s) Y_s' \delta Z_{st} \|_{0,p} + \|  \nabla^2g(Y_s) Y_s' \delta Z_{st} \|_{0,p}  \\
 & \leq \| (\nabla g(Y), \nabla^2 g(Y) Y') \|_{\alpha,Z; L^p(\R^d)} |t-s|^{2 \alpha} \\
 & \qquad + \| \nabla^2 g(Y_s)\|_{0,p} \|Y_s'\|_{C_b(\R^d:\R^d)} \|Z\|_{\alpha} |t-s|^{\alpha} \\
 & \lesssim  \| g\|_{3,p}  \ND(Y)^{-1/2} ( 1 + \|Z\|_{\alpha})^2 (1 +   \|Y_0'\|_{C_b(\R^d)} + \|(Y,Y')\|_{\alpha, Z; C_b(\R^d)})^2 |t-s|^{\alpha} 
\end{align*} 
giving
\begin{align*}
\| \nabla( \delta g(Y)_{st} &  - \nabla g(Y_s) Y_s' \delta Z_{st})\|_{0,p} \\
 \lesssim &  \| g\|_{3,p}  \ND(Y)^{-1/p} ( 1 + \|Z\|_{\alpha})^2 (1 +   \|Y_0'\|_{C_b^1(\R^d)} + \|(Y,Y')\|_{\alpha, Z; C_b^1(\R^d)})^2 |t-s|^{\alpha}  .
\end{align*}
Similarly we get 
\begin{align*}
\| \nabla( \delta (\nabla g(Y) Y')_{st} \|_{0,p} \lesssim   \| g\|_{3,p}  \ND(Y)^{-1/2} ( 1 + \|Z\|_{\alpha})^2 (1 +   \|Y_0'\|_{C_b^1(\R^d)} + \|(Y,Y')\|_{\alpha, Z; C_b^1(\R^d)})^2 |t-s|^{\alpha}  ,
\end{align*}
which shows that 
$$
\|( g(Y) , \nabla g(Y) Y' )  \|_{\alpha, Z; W^{1,p}(\R^d)}  \leq   C \ND(Y)^{-1/p} ( 1 + \|Z\|_{\alpha})^2 ( 1 +  \|Y_0'\|_{C_b^1(\R^d)} + \|(Y,Y')\|_{\alpha, Z; C_b^1(\R^d)})^2  \|g\|_{3,p}. 
$$

The general statement is proved using similar, though more involved computations, and is left to the interested reader.
\end{proof}

It is straightforward to see that the space $ \D^{2 \alpha}_{Z} ([0,T]; \mathcal{L}(E))$ acts as a multiplier on $\D^{2 \alpha}_{Z} ([0,T]; E)$ in the following way:

\begin{lemma} \label{abstractMultiplier}
For $(G,G') \in \D^{2 \alpha}_Z([0,T];E)$ and $(A,A') \in \D^{2 \alpha}_{Z} ([0,T]; \mathcal{L}(E))$ we have $(AG, AG' + A'G) \in \D^{2 \alpha}_Z([0,T]; E)$ with bounds
\begin{equation} \label{abstractMultiplierBound}
\|(AG, AG' + A'G) \|_{\alpha,Z; E} \leq ( \|A'_0\|_{\mathcal{L}(E)} + \|(A,A')\|_{\alpha, Z;\mathcal{L}(E) } ) ( \|G'_0\|_{E} + \|(G,G')\|_{\alpha, Z; E} ) (1 + \|Z\|_{\alpha})^2 . 
\end{equation}
\end{lemma}

Combined with Lemma \ref{L2Composition} we get the following. 
\begin{corollary} \label{L2Multiplier}
Assume $(Y,Y') \in \D_Z^{2 \alpha}([0,T]; C^1_b(\R^d;\R^d))$ satisfies $\zeta(Y) >0$ and $(A,A') \in \D^{2 \alpha}_{Z} ([0,T]; L^{\infty}(\R^d))$. For any $g \in W^{2,p}(\R^d)$ we have
$$
(g(Y)A, \nabla g(Y) Y' A + g(Y) A') \in \D_Z^{2 \alpha}([0,T]; L^p(\R^d))
$$
with the bound
\begin{align*}
\|(g(Y)A, \nabla g(Y) Y' A + g(Y) A')   \|_{\alpha, Z; L^p} &  \leq C \ND(Y)^{-1/p} ( 1 + \|Z\|_{\alpha})^4   ( \|A'_0\|_{L^{\infty}} + \|(A,A')\|_{\alpha, Z; L^{\infty}(\R^d)} ) \\
& \qquad \times (  \|Y_0'\|_{\infty} + \|(Y,Y')\|_{\alpha, Z;C_b(\R^d;\R^d)})^2  \|g\|_{2,p} 
\end{align*}
\end{corollary}

We shall show that the uniform bounds in \eqref{flowDeterminantCondition} holds for flows generated by RDE's. 
\begin{definition} \label{dualLift}
For a rough path $\ZZ$ denote by $\|\ZZ \|_{p-var, [s,t]}$ the homogeneous $p$-variation norm of $\ZZ$ over $[s,t]$. Define inductively
\begin{align*}
\tau_0 &= s \\
\tau_{i+1} & := \inf \{ u : \|\ZZ \|_{p-var, [\tau_i,u]} \geq 1 , \, \, \tau_i < u \leq t \} \wedge t
\end{align*}
and finally 
$$
N_{ [s,t]}(\ZZ) := \sup\{ n \in \mathbb{N} : \tau_n < t \} .
$$
\end{definition}

We shall need the following fact about $N_{[s,t]}(\ZZ)$. For a proof, see \cite{DOR}, \cite[Lemma 36]{DFS}.

\begin{proposition} \label{summary}
For an $\alpha$-H\"{o}lder (deterministic) rough path $\WW = (W, \mathbb{W})$, define the $\alpha$-H\"{o}lder rough path joint lift $\ZZ(\WW) = (Z, \mathbb{Z})$ where 
$$
Z_t = \left(\begin{array}{l}
B_t \\
W_t \\
\end{array} \right)
\hspace{.5cm} and \hspace{.5cm} 
\mathbb{Z}_{st} = \left(\begin{array}{cc}
\mathbb{B}_t & \int_s^t B_{sr} dW_r \\
\int_s^t W_{sr} dB_r & \mathbb{W}_{st} \\
\end{array} \right) .
$$
Then the random variable 
$$
\sup_{ \| \ZZ \|_{\alpha} \leq M} N_{[0,T]}(\ZZ(\WW))
$$
has Gaussian tails. Moreover, we have the continuity
$$
E[ \| \ZZ(\WW) - \ZZ(\tilde{\WW}) \|^p_{\alpha} ] \lesssim \| \WW - \tilde{\WW} \|^p_{\alpha} .
$$
\end{proposition}

We remark that for the rough path constructed above, the random variable $\|\ZZ\|_{p-var}^p$ does not have Gaussian tails, which makes it an ill-suited norm for studying linear RDE's under expectations. However, the next result show how $N_{[0,T]}(\ZZ)$ is an appropriate tool for these equations.

\begin{proposition} \label{summary2}
Assume $\ZZ$ is an $\alpha$-H\"{o}lder 
geometric rough path, $\alpha \in (\frac{1}{3}, \frac{1}{2})$. 

\begin{enumerate}

\item
For a collection of vector fields $V = (V_i)_{i=1}^e$  in $C^3_b(\R^d)$, there exists a unique solution $(Y,Y') = (Y, V(Y))$ to the RDE 
$$
dY = V(Y) d\ZZ
$$
which satisfies
$$
\|(Y,Y')\|_{\alpha, Z} \leq C(1 + \| \ZZ \|_{\alpha})^3 ,
$$
and
$$
\|Y \|_{\frac{1}{\alpha} - var} \leq C(1 + N_{[0,T]}(\ZZ)).
$$

For proofs, see \cite[Proposition 8.3]{FH14} and \cite[Lemma 4, Corollary 3]{FR13} respectively.

\item

For a collection of affine vector fields $V = (V_i)_{i=1}^e$, i.e. $V_i(x) = A_ix + b_i$, there exists a unique solution $(Y,Y') = (Y, V(Y))$ to the RDE 
$$
dY = V(Y) d\ZZ
$$

which satisfies

$$
\|Y\|_{\alpha} \leq C(1 + |y_0|) \|\ZZ\|_{\alpha} \exp\{ C N_{[0,T]}(\ZZ) \} 
$$
and
$$
\|(Y,Y') \|_{Z, \alpha } \leq C\exp\{ C N_{[0,T]}(\ZZ) \} ( \|\ZZ\|_{\alpha}^2 \vee \|\ZZ\|_{\alpha}^3) .
$$
For a proof, see \cite[Lemma 26]{DFS}.

\end{enumerate}

\end{proposition}

\begin{lemma} \label{roughLiouvilleLemma}
Denote by $\Psi$ the flow generated by the RDE
$$
d\Psi_{s,t}(x) = V(\Psi_{s,t}(x)) d\ZZ_t \hspace{.5cm} \Psi_{s,s}(x) = x .
$$
Then there exists $C>0$ such that we have
$$
C^{-1} \exp \{-  C N_{ [0,T]}( \mathbf{Z} ) \} \leq  det( \nabla \Psi_{t,T}^{-1}(x) ) \leq C \exp \{ C N_{ [0,T]}( \mathbf{Z} ) \}
$$

\end{lemma}

\begin{proof}

The inverse flow and the Jacobian satisfies  (see e.g. \cite[Lemma 27]{DFS})

$$
d\Psi_{T - t, T}^{-1}(y) = V_i(\Psi_{T - t, T}^{-1}(y)) d\overleftarrow{\mathbf{Z}}^i_t \hspace{.5cm} \Psi_{T , T}^{-1}(y) = y
$$
and 
$$
d \nabla \Psi_{T - t, T}^{-1}(y) = \nabla V_i(\Psi_{T - t, T}^{-1}(y))  \nabla \Psi_{T - t, T}^{-1}(y) d\overleftarrow{\mathbf{Z}}^i_t \hspace{.5cm} \nabla \Psi_{T , T}^{-1}(y) = I_{d \times d} . 
$$
Denote by $D_t := det(\nabla \Psi_{T - t, T}^{-1}(y))$. Then by the Rough Ito formula, since $ det( \cdot) : \R^{d \times d} \rightarrow \R$ has derivative at $A$ equal to $det(A) Tr(A^{-1} M)$ for $det(A) \neq 0$, we get
\begin{align*}
dD_t & = D_t Tr \left( (\nabla \Psi_{T - t, T}^{-1}(y))^{-1} \nabla V_i(\Psi_{T - t, T}^{-1}(y))  \nabla \Psi_{T - t, T}^{-1}(y) \right) d\overleftarrow{\mathbf{Z}}^i_t  \\
& = D_t (div V_i)(\Psi_{T - t, T}^{-1}(y))  d\overleftarrow{\mathbf{Z}}^i_t 
\end{align*}
where we have used $Tr(AB) = Tr(BA)$.

From \cite[Lemma 26]{DFS} we get
$$
\| D \|_{ \alpha, [0,T]} \leq C \|\mathbf{Z} \|_{\alpha, [0,T]} \exp \{ C N_{ [0,T]}( \mathbf{Z}) \} .
$$

We clearly have $D_0 = 1$ and so 
$$
|D_t| \leq |D_t - D_0| + 1 \leq T^{\alpha} \|D\|_{\alpha, [0,T] } + 1 \leq C \|\mathbf{Z} \|_{\alpha, [0,T]} \exp \{ C N_{[0,T]}( \mathbf{Z}) \}
$$
provided $C$ is large enough. To see the reversed inequality, notice that $D^{-1}$ solves
$$
dD_t^{-1} = - D_t^{-1} (div V_i)(\Psi_{T - t, T}^{-1}(y))  d\overleftarrow{\mathbf{Z}}^i_t, 
$$
so that $|D_t|^{-1} \leq C \|\mathbf{Z} \|_{\alpha, [0,T]} \exp \{ C N_{ [0,T]}( \mathbf{Z}) \}$ by similar reasoning as above.
\end{proof}

\begin{lemma}
Suppose $(Y,Y') = (Y(x), Y(x)') = (\Psi_{T - \cdot, T}(x)^{-1}, V(\Psi_{T - \cdot, T}(x)^{-1})) $ where $\Psi$ is the flow of the RDE $d\Psi = V(\Psi) dZ$. Then \eqref{flowDeterminantCondition} holds. More specifically, the exists $C>0$ such that 
$$
C \exp \{-  C N_{  [0,T]}( \mathbf{Z} ) \} \leq  \inf_{s,t \in [0,T], x \in \R^d, \theta \in [0,1]}  det(\nabla ( Y_s(x) + \theta \delta Y_{st}(x) ) ).
$$

\end{lemma}

\begin{proof}

For any $\theta_0 > 0$ we have 
\begin{align*}
\inf_{s,t \in [0,T],  x \in \R^d, \theta \in [0,1]}  & det(\nabla ( Y_s(x) + \theta \delta Y_{st}(x) ) )   \\
 & = \inf_{s,t \in [0,T], x \in \R^d, \theta \in [0, \theta_0]}  det(\nabla ( Y_s(x) + \theta \delta Y_{st}(x) ) ) \\
  & \qquad  \vee \inf_{s,t \in [0,T], x \in \R^d, \theta \in [\theta_0,1]}  det(\nabla ( Y_s(x) + \theta \delta Y_{st}(x) ) ) .
\end{align*}

Step 1: We first choose $\theta_0$ such that 
\begin{equation} \label{uniformDeterminantContinuity}
\inf_{s,t \in [0,T], x \in \R^d, \theta \in [0, \theta_0]}  det(\nabla ( Y_s(x) + \theta \delta Y_{st}(x) ) ) \geq \frac{C}{2} \exp \{-  C N_{  [0,T]}( \mathbf{Z} ) \} .
\end{equation}
To do this, notice first that by Lemma \ref{roughLiouvilleLemma} we have
$$
\inf_{s \in [0,T], x \in \R^d}  det(\nabla  Y_s(x) ) \geq C \exp \{-  C N_{  [0,T]}( \mathbf{Z} ) \} . 
$$
Moreover, we know that $ \delta \nabla Y : [0,T]^2 \times \R^d \rightarrow \R^{d \times d}$ is a bounded and continuous map. Assume now that \eqref{uniformDeterminantContinuity} does not holds, which means that there exists $(s_n,t_n,x_n, \theta_n) \in [0,T]^2 \times \R^d \times [0,1]$ with $\lim_n \theta_n = 0$ such that 
$$
 det(\nabla ( Y_{s_n}(x_n) + \theta_n \delta Y_{s_n t_n }(x_n) ) ) \leq \frac{C}{2} \exp \{-  C N_{  [0,T]}( \mathbf{Z} ) \} + \frac{1}{n}
$$
By the boundedness of $\delta \nabla Y$ we have for a sub-sequence that $\lim_k \nabla  Y_{s_{n_k}}(x_{n_k}) = M \in \R^{d \times d}$ and $\lim_k \delta \nabla  Y_{s_{n_k} t_{n_k}}(x_{n_k}) = N \in \R^{d \times d}$. By the continuity of the determinant we get
\begin{align*}
\frac{C}{2} \exp \{-  C N_{  [0,T]}( \mathbf{Z} ) \} & \geq \lim_k det(\nabla ( Y_{s_{n_k}}(x_{n_k}) + \theta_{n_k} \delta Y_{s_{n_k} t_{n_k} }(x_{n_k}) ) ) = det(M)  \\
 & = \lim_k det(\nabla  Y_{s_{n_k}}(x_{n_k} ) ) \geq \inf_{s \in [0,T], x \in \R^d}  det(\nabla  Y_s(x) ) \geq C \exp \{-  C N_{  [0,T]}( \mathbf{Z} ) \}
\end{align*}
which is a contradiction. 

Step 2: We show that for any $\theta_0 > 0$ we have
$$
\inf_{s,t \in [0,T], x \in \R^d, \theta \in [\theta_0,1]}  det(\nabla ( Y_s(x) + \theta \delta Y_{st}(x) ) )  \geq \frac{C}{2} \exp \{-  C N_{  [0,T]}( \mathbf{Z} ) \}.
$$

For fixed $\theta$ define $Y_t^{s,\theta} : = Y_s + \theta \delta Y_{st}$, for $t \geq s$ so that  
\begin{align*}
d Y_t^{s,\theta} = V_{\theta , j}( Y_t^{s,\theta}) d\overleftarrow{Z}^j_t , \hspace{.5cm} Y_s^{s, \theta} = Y_s  
\end{align*}
where $V_{\theta}(y) := \theta V(Y_s +  \theta^{-1} (y - Y_s))$. As in the proof of the previous lemma, we get
$$
det(\nabla Y_t^{s,\theta} ) = det( \nabla Y_s ) \exp \left\{ tr \left( \int_0^t \nabla V_{\theta , j}(Y_r^{s,\theta} ) dZ^j_r \right)   \right\} =  det( \nabla Y_s ) \exp \left\{   \int_0^t div(V_{\theta,j})(Y_r^{s,\theta} ) dZ^j_r    \right\}. 
$$

Since $V_{\theta}$ is bounded in $C^3(\R^d)$ uniformly in $\theta \in [\theta_0,1]$, the result follows.
\end{proof}

\begin{lemma} \label{LemmaL2Composition}
If $\Psi_{t,s}(x)$ is the flow of the RDE
$$
d\Psi_{t,s}(x) =  V(\Psi_{t,s}(x)) d\ZZ_s , \, \Psi_{t,t}(x) = x
$$ 
then for any $g \in W^{2,p}(\R^d)$ the pair
$$
\left( g( \Psi_{\cdot,T}^{-1}( \cdot )) det( \nabla \Psi_{ \cdot,T}^{-1}( \cdot )) , div(g V_j)(\Psi_{\cdot ,T}^{-1})  det(\nabla \Psi_{ \cdot ,T}^{-1})  \right) \in \mathscr{D}_{Z}^{2 \alpha}([0,T]; L^p(\R^d))
$$
and we have the estimate

\begin{align*}
\big\| \Big(g(\Psi_{ \cdot,T}^{-1}) & det( \nabla \Psi_{ \cdot,T}^{-1}),   div(g V_j)(\Psi_{ \cdot,T}^{-1})  det(\nabla \Psi_{ \cdot,T}^{-1}) \big) \big\|_{\alpha, Z; L^p(\R^d)} \\ 
& \leq  C \exp \{ C N_{  [0,T]}( \ZZ) \}  ( 1 + \| \ZZ\|_{\alpha})^{k}(1 + \|V\|_{\infty})^2 \|g\|_{2,p}
\end{align*} 
for some constants $C$ and $k$.

\end{lemma}

\begin{proof}
%Begin by noticing
%$$
%\int h( \Psi_{t,T}(x)) g(x)  dx = \int h(y) g( \Psi_{t,T}^{-1}(y)) det( \nabla \Psi_{t,T}^{-1}(y))   dy
%$$
%where we have used the change of variables $x = \Psi_{t,T}^{-1}(y)$. 
Since $\Psi_{T - \cdot,T}^{-1}(y)$ satisfies the equation
$$
d_s\Psi_{T -s,T}^{-1}(y) = V(\Psi_{T - s,T}^{-1}(y)) d\ZZ_{T - s} .
$$
If we define $A := g( \Psi_{T - \cdot,T}^{-1})$ and $(A')^j = \nabla g( \Psi_{T - \cdot,T}^{-1}) V_j( \Psi_{T - \cdot,T}^{-1})$
we have by Lemma \ref{L2Composition} that 
$$
(A,A') \in \D_{\overleftarrow{Z}}^{2 \alpha}([0,T]; L^p(\R^d))
$$
where $\overleftarrow{Z}_t := Z_{T-t}$.

Moreover, with $B_r := det( \nabla \Psi_{T-r,T}^{-1})$ and $(B_r')^j = det( \nabla \Psi_{T - r,T}^{-1}) div(V_j)(\Psi_{r,T}^{-1})$ we have 
$$
(B,B') \in \D^{2 \alpha}_{\overleftarrow{Z}} ([0,T]; L^{\infty}(\R^d)) . 
$$

Notice now that 
$$
A_r(B_r')^j + (A_r')^jB_r = div(g V_j)(\Psi_{T- r,T}^{-1}) det(\nabla \Psi_{T - r,T}^{-1})
$$

By Corollary \ref{L2Multiplier} we have
$$
\left( g(\Psi_{T - \cdot,T}^{-1})det( \nabla \Psi_{T - \cdot,T}^{-1}), div(g V_j)(\Psi_{T- \cdot,T}^{-1}) det(\nabla \Psi_{T - \cdot,T}^{-1}) \right) \in \D_{\overleftarrow{Z}}^{2 \alpha}([0,T]; L^p(\R^d)) ,
$$
and if we reverse the time we get by \cite[Lemma 30]{DFS}

\begin{align*}
\big\| \Big(g(\Psi_{ \cdot,T}^{-1}) & det( \nabla \Psi_{ \cdot,T}^{-1}),   div(g V_j)(\Psi_{ \cdot,T}^{-1})  det(\nabla \Psi_{ \cdot,T}^{-1}) \big) \big\|_{\alpha, Z; L^p(\R^d)} \\ 
 & \leq \| \left(g(\Psi_{T - \cdot,T}^{-1})det( \nabla \Psi_{T - \cdot,T}^{-1}), div(g V_j)(\Psi_{T- \cdot,T}^{-1}) det(\nabla \Psi_{T - \cdot,T}^{-1}) \right) \|_{\alpha, \overleftarrow{Z}; L^p(\R^d)} (1 + \| \ZZ \|_{\alpha}) \\
& \leq  C \exp \{ C N_{  [0,T]}( \ZZ) \}  ( 1 + \| \ZZ\|_{\alpha})^k  \\
& \times ( \|B'_0\|_{L^{\infty}(\R^d)} + \|(B,B')\|_{\alpha, Z; L^{\infty}(\R^d)} ) (  \|V\|_{\infty} + \|(\Psi_{T - \cdot,T}^{-1},V(\Psi_{T - \cdot,T}^{-1}))\|_{\alpha, Z})^2 \|g\|_{2,p}  \\
& \leq  C \exp \{ C N_{  [0,T]}( \ZZ) \}  ( 1 + \| \ZZ\|_{\alpha})^k  \\
& \times ( 1 + \exp \{ C N_{  [0,T]}( \ZZ) \} (\|\ZZ\|_{\alpha}^2 \vee \|\ZZ\|_{\alpha}^3 ) (  \|V\|_{\infty} + (1 + \|\ZZ\|_{\alpha})^3)^2 \|g\|_{2,p} \\
& \leq  C \exp \{ C N_{  [0,T]}( \ZZ) \}  ( 1 + \| \ZZ\|_{\alpha})^{k}(1 + \|V\|_{\infty})^2 \|g\|_{2,p} .
\end{align*}
Above the constants $C$ and $k$ may vary from line to line.

%{\color{red}Exponents are not correct in $( 1 + \| \WW\|_{\alpha})^{12}$. Find the correct ones on p.21 in \cite{DFS}.}
%
%{\color{red}Better idea: In fact, do not write down the exponents explicitly, it will only lead to confusion and in the end we don't care if it is 17 +d. Write instead some exponent $k$ that may change from line to line.}

\end{proof}

%Try this:
%
%Since $(\nabla \delta Y_{st}(y))_{s,t,y} $ is a bounded familiy of matrices, the determinant is uniformly continuous on closed bounded sets, for every $\epsilon > 0$ there exists  $\delta > 0$ such that $ |   \theta (\nabla \delta Y_{st}(y)) | < \delta$ implies
%$$
%|det( \nabla (Y_s(y) + \theta \delta Y_{st}(y))) - det( \nabla Y_s(y) ) | \leq \epsilon
%$$
%and
%$$
%det( \nabla (Y_s(y) + \theta \delta Y_{st}(y))) > 0
%$$
%
%\begin{align*}
%\end{align*}

%We move on the generalize Corollary \ref{L2Multiplier}. The proof is immediate from Lemma \ref{abstractMultiplier} and Lemma \ref{L2CompositionH3}.
%\begin{corollary} \label{H3Multiplier}
%Let $(Y,Y') \in \D_Z^{2 \alpha}([0,T]; C^3_b(\R^d;\R^d))$ satisfy the assumption as in Lemma \ref{L2Composition} and $(A,A') \in \D^{2 \alpha}_{Z} ([0,T]; C^3(\R^d))$. For any $g \in H^5$ we have
%$$
%(g(Y)A, \nabla g(Y) Y' A + g(Y) A') \in \D_Z^{2 \alpha}([0,T]; 3,p)
%$$
%with the bound
%\begin{align*}
%\|(g(Y)A, \nabla g(Y) Y' A + g(Y) A') &  \|_{\alpha, Z;3,p}  \leq C \ND(Y)^{-1/2} ( 1 + \|Z\|_{\alpha})^2  \\
%& \times ( \|A'_0\|_{C^3} + \|(A,A')\|_{\alpha, Z; L^{\infty}(\R^d)} ) (  \|Y_0'\|_{\infty} + \|(Y,Y')\|_{\alpha, Z})^2( |\nabla g|_{L^p}  + |\nabla^2 g|_{L^p}) 
%\end{align*}
%
%
%\end{corollary}
%

\end{document}